%% file: squareMult.tex
\documentclass{timgad} 
\pdfoutput=1

\usepackage{float}

\title{Square-free graphs are multiplicative}
\date{First appeared: January 2016. Updated: August 2016.}
\author{Marcin Wrochna}
\affil{%
	Institute of Informatics, University of Warsaw, Poland\\
	\textnormal{\texttt{m.wrochna@mimuw.edu.pl}}}
\keywords{square-free, four-cycle-free, homomorphism, Hedetniemi's conjecture, homotopy, multiplicativity, box complex}

\usepackage[backend=bibtex,bibencoding=ascii,style=alphabetic,bibstyle=alphabetic,firstinits=true,doi=false,isbn=false,url=false,maxbibnames=99]{biblatex}
\renewbibmacro{in:}{}
\AtEveryBibitem{\clearname{editor}}

\newbibmacro{string+link}[1]{%
	\iffieldundef{ee}
	{\iffieldundef{url}
		{\iffieldundef{doi}
			{\iffieldundef{eprint}
				{#1}
				{\href{http://arxiv.org/abs/\thefield{eprint}}{#1}}}
			{\href{http://dx.doi.org/\thefield{doi}}{#1}}}
		{\href{\thefield{url}}{#1}}}
	{\href{\thefield{ee}}{#1}}} 
\DeclareFieldFormat{title}{\usebibmacro{string+link}{\mkbibemph{#1}}}
\DeclareFieldFormat[article,inproceedings,inbook,incollection,mastersthesis,thesis]{title}{\usebibmacro{string+link}{\mkbibquote{#1}}}

\def\NN{\mathbb{N}}
\def\ZZ{\mathbb{Z}}
\def\Hom{\mbox{Hom}}
\def\col{\mu}
\def\eps{\varepsilon}

\def\red#1{\overline{#1}}
\def\sed#1{[#1]}
\def\cat{\_}
\def\pexp#1#2{{#1}^{\cdot #2}}
\def\od{\vec{d}}
\def\pd{\Delta}

\def\grouppi{\boldsymbol{\pi}}
\def\fundPiSymbol{\boldsymbol{\pi}}
\def\fundPi#1{\boldsymbol{\pi}(#1)}
\def\fundpi#1#2{\boldsymbol{\pi}_{#1}(#2)}
\def\coarsePiSymbol{\boldsymbol{\pi}_{/\sim}}
\def\coarsePi#1{\boldsymbol{\pi}(#1)_{/\sim}}
\def\coarsepi#1#2{\boldsymbol{\pi}_{#1}(#2)_{/\sim}}

\def\lcycle#1#2{#1 \otimes #2}
\def\rcycle#1#2{#1 \otimes #2}

\usepackage{tikz}
\usetikzlibrary{shapes,backgrounds,plothandlers,plotmarks,calc,arrows,fadings,patterns,decorations.pathmorphing,decorations.pathreplacing,fadings,decorations.markings}

\tikzset{ %
	H/.style={circle,fill=gray!20,draw=gray!20,line width=2pt,inner sep=0pt,minimum size=15pt},
	He/.style={draw=gray!20,line width=12pt},
	G/.style={circle,fill=black,inner sep=0pt,minimum size=3pt},
	G1/.style={G,minimum size=5pt},
	Ge/.style={draw=black},
	v/.style={circle,draw=black!75,inner sep=0pt,minimum size=8pt},
	v1/.style={v,line width=1.5pt},
	Gv/.style={inner sep=0pt},
	Gve/.style={Ge,dashed},
}

\begin{document}

\maketitle

\begin{abstract}
A graph $K$ is \emph{square-free} if it contains no four-cycle as a subgraph (i.e., for every quadruple of vertices, if $v_0v_1,v_1v_2,v_2v_3,v_3v_0 \in E(K)$, then $v_0=v_2$ or $v_1=v_3$).
A graph $K$ is \emph{multiplicative} if $G\times H \to K$ implies $G\to K$ or $H \to K$, for all graphs $G,H$.
Here $G\times H$ is the tensor (or categorical) graph product and $G\to K$ denotes the existence of a graph homomorphism from $G$ to $K$.
Hedetniemi's conjecture, which states that $\chi(G\times H)=\min(\chi(G),\chi(H))$, is equivalent to the statement that all cliques $K_n$ are multiplicative.
However, the only non-trivial graphs known to be multiplicative are $K_3$, odd cycles, and still more generally, circular cliques $K_{p/q}$ with $2\leq \frac{p}{q}<4$.
We make no progress for cliques, but show that all square-free graphs are multiplicative.
In particular, this gives the first multiplicative graphs of chromatic number higher than~4.
Generalizing, in terms of the \emph{box complex}, the topological insight behind existing proofs for odd cycles, we also give a different proof for circular cliques with $2\leq\frac{p}{q}<4$.
\end{abstract}

\section{Introduction}\label{sec:intro}
\input{intro}

\section{Preliminaries}\label{sec:prelims}
\input{prelims}

\section{Topological invariants of cycles}\label{sec:commute}
\input{commute}

\section{The case when \texorpdfstring{$K$}{K} is circular}\label{sec:circular}
\input{circular}

\section{The case when \texorpdfstring{$K$}{K} is square-free}\label{sec:recoloring}
\input{recoloring}

\section{Conclusions}\label{sec:conclusions}
\input{concls}

\subsection{Acknowledgments}
The author would like to thank Szymon Toruńczyk for many helpful discussions.

\printbibliography

\end{document}

%% file: intro.tex
Graph homomorphism (see Section~\ref{sec:prelims} for definitions) is an ubiquitous notion in graph theory with a variety of applications, and a first step to understanding more general constraints given by relational structures, see e.g. the monograph of Hell and Ne{\v{s}}et{\v{r}}il~\cite{hell2004homomorphism_book}.
In this work, we consider undirected, simple graphs without loops (the questions we consider have trivial answers for graphs with loops).
The tensor product $G\times H$ is a natural operation arising in this context (in particular, it coincides with the so called categorical product in the category of graphs).
Multiplicative graphs are then graphs $K$ such that $G\times H \to K$ implies $G\to K$ or $H\to K$, for all graphs $G,H$.
That is, the product of any two graphs admits a homomorphism to $K$ only if at least one of the factors does (the converse is trivial, as $G\times H \to G$ and $G\times H \to H$).
The tensor product and disjoint union define a distributive lattice structure on the category of graphs, in which multiplicative graphs, coinciding with so called \emph{meet-irreducible} elements of the lattice, play a crucial role (see~\cite{hell79Category,hell2004homomorphism_book,DuffusS96}).
However, most interest in multiplicativity originates in Hedetniemi's conjecture~\cite{hedetniemi1966homomorphisms}, which states that all cliques $K_n$ are multiplicative. This is equivalent to the original simple statement on the chromatic number, namely ``$\chi(G\times H)=\min(\chi(G),\chi(H))$ for any two graphs $G,H$''.
See~\cite{tardif2008hedetniemi,Sauer01,zhu1998survey} for surveys.

That $K_2$ is multiplicative---i.e., a product of two graphs is bipartite iff one of the factors is---follows easily from the fact that a graph is bipartite iff it has no odd-length cycle.
$K_3$ was proved to be multiplicative by El~Zahar and Sauer~\cite{El-ZaharS85}.
Their proof was generalized to odd cycles by H{\"{a}}ggkvist et~al.~\cite{HaggkvistHMN88}.
Much later, Tardif~\cite{Tardif05} used these results together with general constructions on graphs to show that circular cliques $K_{p/q}$ are multiplicative for any integers $p,q$ satisfying $2\leq \frac{p}{q}<4$.
Circular cliques generalize cliques and odd cycles in the sense that $K_n = K_{n/1}$, $C_{2n+1} = K_{(2n+1)/n}$, and $K_{p/q} \to K_{p'/q'}$ holds iff $\frac{p}{q} \leq \frac{p'}{q'}$ (see~\cite{Zhu01}).

These are the only undirected graphs known to be multiplicative (up to `homomorphic equivalence', note e.g. that even cycles admit homomorphisms from and to $K_2$, multiplicativity follows trivially).
To show a typical example of graphs that are not multiplicative, consider two graphs $G,H$ such that $G\not\to H$ and $H\not\to G$; then $K=G\times H$ is not multiplicative, because $G\times H \to K$ trivially holds, whereas $G\to K$ would imply $G \to G \times H \to H$ (and $\to$ is transitive).

Many other partial results in different directions are known.
A relatively recent one is a proof by Delhomm{\'{e}} and Sauer~\cite{DelhommeS02} that if $G$ and $H$ are connected graphs each containing a triangle and $K$ is a square-free graph, then $G\times H \to K$ implies $G\to K$ or $H\to K$.
Square-free graphs are graphs without $C_4$ as a subgraph (not necessarily induced; if we allowed loops, we should exclude two adjacent loops and a triangle with a looped vertex, too).

The purpose of this paper is to show that in fact all square-free graphs are multiplicative, answering a question of~\cite{tardif2008hedetniemi}.
Since graphs of girth at least 5 are square-free and are well known to have arbitrarily high chromatic number, this gives in particular the first multiplicative graphs of chromatic number higher than 4 (still, this does not imply any new case of Hedetniemi's conjecture).
We also give a different proof of the multiplicativity of $K_{p/q}$ for $2\leq \frac{p}{q} < 4$, very similar to the earlier proofs for odd cycles~\cite{El-ZaharS85,HaggkvistHMN88} and~\cite{DelhommeS02}, by making the topological intuitions therein more general and explicit.
This shows that this approach applies to all graphs currently known to be multiplicative.
However, it makes it all the more interesting to ask whether this can be connected with the approach in~\cite{Tardif05}, which appears to be more general.

\subparagraph{Topological combinatorics}
Formally, all the proofs in this paper are self-contained and combinatorial, not requiring any knowledge of topology.
However, the intuitions behind proofs heavily rely on some basic algebraic topology.

The name of \emph{topological combinatorics} may at first seem oxymoronic, but it is in fact now an established branch originating in Lov{\'{a}}sz' surprising topological proof~\cite{Lovasz78} of a conjecture of Kneser.
This approach was further explored, giving various related \emph{topological lower bounds} on the chromatic number of a graph.
See~\cite{matousek2008using} for a graceful and extensive introduction to the topic, or~\cite{SimonyiZ10} for a faster one.
Without going into much detail, the idea is to define, for every graph $G$, a topological space (more exactly, a simplicial complex) $B(G)$ called the \emph{box complex} 
with the property that $G\to K$ implies the existence of a non-trivial continuous map from $B(G)$ to $B(K)$. Comparing topological properties of the topological spaces one can show, e.g. using the Borsuk-Ulam theorem, that such a continuous map is not possible, and hence $G \not \to K$.
Results of Csorba~\cite{Csorba08} and Kozlov~\cite{kozlovChapter} imply that $B(G\times H)$ is equivalent to $B(G)\times B(H)$, from which it follows easily that if $G,H$ are graphs for which topological lower bounds on the chromatic number are tight, then $G\times H \to K_n$ implies $G\to K_n$ or $H\to K_n$ for all $n$; see Remark~3 in~\cite{SimonyiZ10} for details.

The proofs on multiplicativity in~\cite{El-ZaharS85,HaggkvistHMN88,DelhommeS02} rely on a common topological invariant, which is essentially just the parity of the winding number of certain cycles.
This bears no relation to the above research direction.
We will show that a stronger invariant can be used, essentially the homotopy type of these cycles.
The topological space that we think of here is formed from a graph by taking a copy of the $[0,1]\subseteq\mathbb{R}$ interval for every edge and identifying all endpoints that correspond to the same vertex (in other words, we view the graph as a simplical complex).
Consider now paths in this topological space up to homotopy, that is, up to continuous transformations that fix the endpoints. Paths with a matching endpoint can be concatenated, giving a new path.
This gives a group-like structure called the \emph{fundamental groupoid}, which we will denote as $\fundPi{G}$.
A precise definition with examples will come later.

It turns out that for graphs with squares, a topological space with better properties is obtained by gluing every square $v_0,v_1,v_2,v_3$ in the graph as a face to the cycle it forms.
So a path going from $v_0$ through $v_1$ to $v_2$ can be continuously transformed to go through $v_3$ instead (in other words, we add each square to the simplicial complex).
This will give a coarser groupoid $\coarsePi{G}$ in which more paths are considered equivalent.
We also rely on parity arguments involving the lengths of paths, which could be thought of in topological terms by starting from the graph $G\times K_2$ instead.\looseness=-1

At this point, a reader familiar with the box complex $B(G)$ may realize that we are essentially again describing $B(G)$ as the topology corresponding to $G$. Indeed, one of several definitions of $B(G)$ involves taking $G\times K_2$ as a simplicial complex and adding all complete bipartite subgraphs as faces, squares in particular (since we only look at the fundamental group afterwards, higher dimensional faces turn out to be irrelevant).
However, we mostly use it to describe square-free graphs $K$ or other graphs in which topological lower bounds are almost trivial (the homotopy type of box complexes of square-free graphs has been described exactly in~\cite{Kamibeppu}).
The observation that $B(G)$ will describe the intuitions well may thus be just a coincidence.
Still, it would be a particularly intriguing coincidence, and suggests many possible approaches to generalizations.
On the other hand, the results still strongly rely on analyzing odd cycles, which may exclude interesting generalizations, in particular any applicability to Hedetniemi's conjecture.

\subparagraph{Recoloring and graph homomorphism spaces}
When a graph $K$ is fixed, we can think of its vertices as \emph{colors} and we call a graph homomorphism $G \to K$ a $K$-coloring of $G$.
\emph{Recoloring} a $K$-coloring $\mu$ then means changing the colors assigned by $\mu$ one vertex at a time, maintaining the property of being a $K$-coloring at all times.
Recently the problem of whether two given $K$-colorings can be reached one from the other by recoloring has been analyzed from an algorithmic point of view.
This is part of a larger framework called \emph{reconfiguration}, in which one looks for sequences of simple transformations between solutions of combinatorial problems, see e.g.~\cite{Schwerdtfeger15,Heuvel13}. 

Recoloring defines a graph whose vertices are all $K$-colorings of $G$ and whose edges correspond to a single recoloring step.
Essentially the same graph is known elsewhere as the 1-skeleton of $\Hom(G,K)$ (a certain simplicial complex having $K$-colorings of $G$ as vertices, see e.g.~\cite{Dochtermann09}; let us just mention that $\Hom(K_2,G)$ happens to be another definition of the box complex of $G$).
Another object describing a `space of $K$-colorings' is the \emph{exponential graph} $K^G$, with all functions $V(G)\to V(K)$ as vertices and an edge between two functions $f,g$ whenever the pair defines a $K$-coloring of $G\times K_2$.
The exponential graph is especially relevant in context of Hedetniemi's conjecture; e.g., a graph $K$ is multiplicative if and only if for every graph $G$, either $K^G$ or $G$ is $K$-colorable.
Since $G \times K^G \to K$ can be easily shown, this means the exponential graph provides the hardest test for multiplicativity, in a sense; see e.g.~\cite{tardif2008hedetniemi}.
Paths in $K^G$ are also related to recoloring, as it is easily checked that two edges $fg$ and $f'g'$ of $K^G$ are connected by a path if and only if the corresponding $K$-colorings of $G\times K_2$ can be recolored into each other.

The aformentioned algorithmic problem concerning recoloring has been shown to be, somewhat surprisingly, solvable in polynomial time for $K=K_3$~\cite{CerecedaHJ11} (it is  \textsc{PSPACE}-complete for $K=K_n$ with $n\geq 4$~\cite{BonsmaC09}).
The author~\cite{Wrochna15} generalized the result, giving a polynomial algorithm for any square-free graph $K$.
This is achieved by the following characterization: two $K$-colorings of a graph $G$ can be recolored into one another (are in the same connected component of $\Hom(G,K)$) if and only if they are topologically equivalent (i.e. homotopically equivalent as maps between 1-dimensional simplicial complexes) and two more simple conditions are satisfied.
In this work, we cannot use this characterization directly, but we use the same topological tools and the general intuition that if we can do something via continuous transformations, then we can probably do this via recoloring, at least in square-free graphs.

The completeness result of~\cite{BonsmaC09} suggests that useful topological invariants might not exist for $K=K_4$.
On the other hand, it is also possible that such invariants do exist, but are algorithmically more complex only because they involve higher dimensions or additional discrete conditions.
In any case, studying this problem  from an algorithmic point of view for other $K$ might elucidate possible extensions of the topological approach, at least to graphs $K$ with a `1-dimensional' box complex.
An algorithm for $K=K_{p/q}$ for $2\leq \frac{p}{q} <4$ was recently shown by Brewster et al.~\cite{BrewsterMMN15}.

Interesting, partly related properties of homomorphisms from infinite grids $\ZZ^d$ to square-free graphs, with motivations in statistical thermodynamics, were found independently by Chandgotia~\cite{Chandgotia14}.

\subparagraph{Proof outline}
In Section~\ref{sec:prelims} we define basic notions and the fundamental groupoid $\fundPi{G}$.
Section~\ref{sec:commute} defines the equivalence relation $\sim$ between walks and the coarse groupoid $\coarsePi{G}$, which can be thought as quotienting $\fundPi{G}$ by the squares of $G$.
Then, basic facts about it are proved: that $\fundPi{G}$ is isomorphic to $\coarsePi{G}$ for square-free $G$, that graph homomorphism induce groupoid homomorphisms (in category-theoretic terms, $\coarsePi{\cdot}$ is a functor), and that $\coarsePi{G \times H}$ is almost isomorphic to $\coarsePi{G} \times \coarsePi{H}$, in a sense.
This last fact is used to show, for $\mu:G\times H \to K$, that: (1.) closed walks in $K$ coming from odd cycles in $G$ and $H$ must wind around the same cycles of $K$ (precise definitions will come in the main text), and (2.)  such cycles can be composed into one coming from a common cycle in $G\times H$ of odd length.

Section~\ref{sec:circular} then considers the case $K$ is a circular clique, first showing that indeed $K$ is topologically a circle ($\coarsePi{K}$ is isomorphic to $\ZZ$).
Then, from the above (2.) it will easily follow that closed walks in $K$ coming from odd cycles in $G$ and $H$ will wind twice an odd and even (respectively) number of times, or vice versa.
We conclude that in $G$, say, all odd cycles have odd winding `half-parity'.
This odd parity then implies that every odd cycle in $G$ has an edge $g_0g_1$ such that $\mu$ maps the edge $(g_0,h_0)(g_1,h_1)$ of $G\times H$ close to its antipode $(g_0,h_1)(g_1,h_0)$ (for some fixed edge $h_0h_1$ of $H$).
Such edges of $G$ can be disregarded, and we get a bipartite subgraph of $G$, which we color with either $\mu(\cdot,h_0)$ or $\mu(\cdot,h_1)$ according to a bipartition.
This gives a $K$-coloring of $G$, concluding the proof for circular cliques.

For the case of square-free $K$, in the last theorem of Section~\ref{sec:commute}, we show that the above (1.) implies that either all closed walks in $K$ coming from cycles in $G$ are topologically trivial (that is, they reduce to $\eps$), or the same holds for $H$ instead, or all cycles in $G\times H$ map to closed walks winding around the same root of $K$.
This is the starting point for Section~\ref{sec:recoloring}, where we aim to improve the $K$-coloring so that it corresponds to its topological type more closely.
An improved $K$-coloring of $G\times H$ then turns out to be in fact just a $K$-coloring of $G$ or $H$ composed with a projection, or a graph homomorphism from $G\times H$ to a cycle in $K$ (the cycle, actually a closed walk in general, corresponding to the root of $K$ that all of $G\times H$ winds around). 
Multiplicativity of cycles then implies $G\to K$ or $H\to K$.

%% file: prelims.tex
\subparagraph{Graphs}
An (undirected, simple) \emph{graph} $G$ is a pair $(V(G),E(G))$, where $V(G)$ is a finite set of \emph{vertices} and $E(G) \subseteq \{\{u,v\} \mid u\neq v\in V(G)\}$ is a set of \emph{edges}.
For \marginpar{$N_G()\\N_G^2()$}
a vertex set $S\subseteq V(G)$, its \emph{neighborhood} $N_G(S)$
is defined as $\{v \mid \{u,v\}\in E(G), u\in S\}$.
We write $N_G(v)$ for $N_G(\{v\})$ and $N_G^2(S)$ for $N_G(N_G(S))\setminus S$. Two vertices $u,v$ are called \emph{adjacent} if $\{u,v\}\in E(G)$.

Throughout this paper we assume that all graphs have at least two vertices and every vertex has a neighbor (that is, there are no isolated vertices---otherwise we could handle them trivially).

\smallskip
The path $P_n$ \marginpar{$P_n, C_n$\\$K_n$}
is the graph with $V(P_n)=\{1,\dots,n\}$ and $E(P_n)=\{\{i,i+1\} \mid i=1\dots n-1\}$.
The~cycle $C_n$ is the graph with $V(C_n)=\{0,\dots,n-1\}$ and $E(C_n)=\{\{i,i+1 \mod n\} \mid i=0\dots n-1\}$.
The~clique $K_n$ is the graph with $V(K_n)=\{0,\dots,n-1\}$ and $E(K_n)=\{\{i,j\} \mid i\neq j \in V(K_n)\}$.
The circular clique $K_{p/q}$, \marginpar{$K_{p/q}$}
for integers $p,q$ such that $\frac{p}{q}\geq2$, is the graph with $V(K_{p/q})= \ZZ_p$ and $E(K_{p/q}) = \{ \{i,i+j\} \mid j=q,q+1,\dots,p-q,\ i\in\ZZ_p\}$.

\subparagraph{Graph homomorphism, product, projection}
For \marginpar{$G\to K$}
two graph $G, K$, a homomorphism $\mu$ from $G$ to $K$ is a function $V(G)\to V(K)$ such that any $\{g_0,g_1\}\in E(G)$ implies $\{\mu(g_0),\mu(g_1)\}\in E(K)$. We write $\mu:G\to K$ if $\mu$ is a homomorphism from $G$ to $K$ and simply $G\to K$ if such a homomorphism exists.
Two graphs are \emph{homomorphically equivalent} if $G\to K$ and $K\to G$.

For\marginpar{$G\times H$}
two graphs $G,H$, their \emph{tensor product} (also called \emph{categorical product}), denoted $G\times H$ is the graph with $V(G\times H)=V(G)\times V(H)$ and $\{(g_0,h_0),(g_1,h_1)\}\in E(G\times H)$ iff $\{g_0,g_1\}\in E(G)$ and $\{h_0,h_1\}\in E(H)$.
For\marginpar{{\footnotesize{projection}}\\$v|_G$}
a vertex $v=(g,h)\in V(G\times H)$, $v|_G=g$ denotes its projection to $G$.
\pagebreak[1]

A\marginpar{\footnotesize{bipartite}}
graph $G$ is \emph{bipartite} if $V(G)$ can be partitioned into two sets $A,B$ such that $E(G)\subseteq \{ \{a,b\} \mid a\in A, b \in B\}$. Equivalently, $G$ contains no odd cycles. Also equivalently, $G\to K_2$.
Whenever the connectivity of some graph is needed, we frequently use the fact that $G\times H$ is connected if and only if $G,H$ are connected and at least one of them is not bipartite.
For $h_0h_1\in H$, \marginpar{$G\times h_0h_1$}
we write $G\times h_0h_1$ for the subgraph of $G\times H$ induced by $V(G) \times \{h_0,h_1\}$, isomorphic to $G\times K_2$.
Note that $C_n \times K_2$ is a cycle for $n$ odd.

\subparagraph{Walks}
Instead of continuous paths we will simply consider `walks' in our graphs.
Fix a graph $G$.\marginpar{$uv \in G$}
An \emph{oriented edge} is an ordered pair $e=(u,v)$ such that $\{u,v\}$ is an edge of $G$; we will always write oriented edges $(u,v)$ of $G$ as $uv\in G$ for short.
We denote the initial and terminal vertex of an oriented edge $e=uv$ as $\iota(e)=u$ and $\tau(e)=v$, respectively.
We write $e^{-1}$ for the oriented edge $\tau(e)\iota(e)$.
A \emph{walk} \marginpar{\footnotesize{walk}}
from $u$ to $v$ is any sequence of oriented edges $e_1 \cat e_2 \dots \cat e_n$ such that $\iota(e_1)=u, \tau(e_n)=v,$ and $\tau(e_i)=\iota(e_{i+1})$ for $i=1\dots n-1$ (the edges are not necessarily distinct).

A walk is \emph{closed} if $\iota(e_1) = \tau(e_n)$.
The \marginpar{$|W|$, $\eps$}
length of a walk, denoted $|W|$, is the number of edges in it.
We write $\eps$ for an empty walk (of length zero; formally there is a different empty walk for every possible starting vertex, but the endpoints of $\eps$ are clear from context).
For \marginpar{$W \cat W'$\\$W^{-1}$}
two walks $W=e_1 \cat \dots \cat e_n, W'=e'_1 \cat \dots \cat e'_m$ with $\tau(e_n)=\iota(e'_1)$, we write $W \cat W'$ for their concatenation $e_1 \cat \dots \cat e_n \cat e'_1 \cat \dots \cat e'_m$ and $W^{-1}$ for the reverse walk $e_n^{-1}\cat \dots \cat e_1^{-1}$. We identify an edge with a walk of length 1 by abuse of notation.

Note \marginpar{$\mu(W)$}
that if $\mu:G\to K$ is a graph homomorphism and $e=uv$ is an oriented edge in $G$, then $\mu(u)\mu(v)$ is an oriented edge in $K$ by definition, which we henceforth denote as $\mu(e)$.
Similarly if $W=e_1 \cat \dots \cat e_n$ is a walk in $G$, then $\mu(e_1) \cat \dots \cat \mu(e_n)$ is a walk in $K$, which we henceforth denote as $\mu(W)$.
In particular \marginpar{$W|_G$}
projections are homomorphisms, $|_G\colon G\times H \to G$, so if $W$ is a walk in $G\times H$, then $W|_G$ is a walk in $G$.
 
\subparagraph{Reducing, the fundamental groupoid}
Homotopy corresponds to the following notion for walks.
We call a walk \emph{reduced} if it contains no two consecutive edges $e_i,e_{i+1}$ such that $e_{i+1}=e_i^{-1}$ (in other words, it never backtracks).
\emph{Reducing} a walk means deleting any such two consecutive edges from the sequence.
It \marginpar{$\red{W}$}
can easily be seen that by arbitrarily reducing a walk $W$ until no more reductions are possible, one always obtains the same reduced walk, which we therefore denote as $\red{W}$, see Figure~\ref{fig:reducing}.
For \marginpar{$W \cdot W'$}
any two reduced walks $W$ from $u$ to $v$ and $W'$ from $v$ to $w$, we write $W \cdot W'$ for the walk from $u$ to $w$ obtained by their concatenation and reduction: $\red{W \cat W'}$.

\begin{figure}[H]
	\centering
	\input{figReducing}
	\vspace*{-0.5em}	
	\caption{Examples of two walks (in a graph on 10 vertices) which reduce to the same, bottom left one. The bottom right one is a different reduced walk; when its endpoints are fixed, it cannot be distorted as a curve to give any of the others.}	
	\label{fig:reducing}
\end{figure}
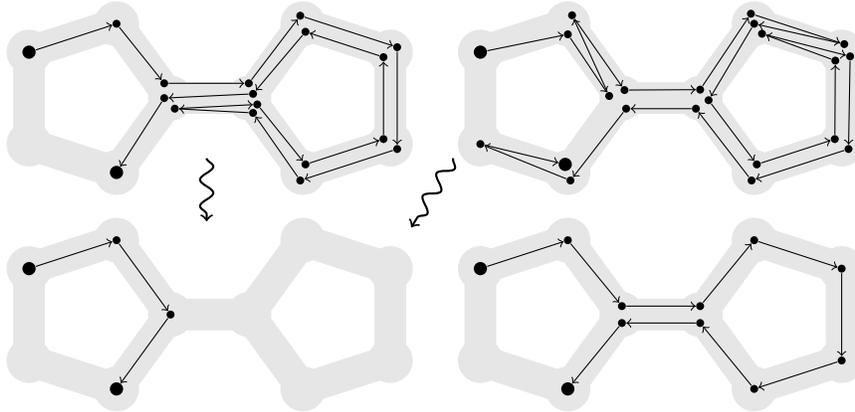
\pagebreak[2]

The set of all reduced walks in a graph (which is infinite, except for trees), together with inversion $^{-1}$ and concatenation $\cdot$, forms a group-like structure, see Figure~\ref{fig:groupoid}.
Formally, a \emph{groupoid} \marginpar{\footnotesize{groupoid}}
is a set $\Pi$ of \emph{elements} with a unary operation ${}^{-1} : \Pi \to \Pi$ and a partial binary operation $\cdot : \Pi \times \Pi \rightharpoonup \Pi$ (not necessarily defined for every pair of elements) that satisfies group axioms.
These are associativity (meaning if $P \cdot Q$ and $Q\cdot R$ are defined, then $P \cdot (Q \cdot R)$ and $(P\cdot Q) \cdot R$ are defined and equal; conversely if one of these expressions is defined, then both are defined and equal); $P \cdot P^{-1}$ and $P^{-1} \cdot P$ are always defined and equal; and if $P \cdot Q$ is defined, then $P \cdot Q \cdot Q^{-1} = P$ and $P^{-1} \cdot P \cdot Q = Q$.
Standard properties that hold in groups can easily be deduced, including $(P^{-1})^{-1}=P$ and $(P\cdot Q)^{-1} = Q^{-1} \cdot P^{-1}$.
A \emph{groupoid homomorphism} $\phi: \Pi\to \Pi'$ is a function such that $\phi(P^{-1})=\phi(P)^{-1}$, and if $P\cdot Q$ is defined in $\Pi$, then $\phi(P)\cdot \phi(Q)$ is defined and equal to  $\phi(P\cdot Q)$ in $\Pi'$.\looseness=-1

A straightforward check shows that the set of reduced walks in $G$ with operations ${}^{-1}$ and $\cdot$ forms a groupoid.
We \marginpar{$\fundPi{G}$\\$\fundpi{v}{G}$}
call it the \emph{fundamental groupoid} of $G$ and denote it $\fundPi{G}$.
For a vertex $v\in V(G)$, we write $\fundpi{v}{G}$ for the group formed by the same operations on closed reduced walks from $v$ in $G$ (and hence ending in $v$ too).
These are the same definitions as in e.g.~\cite{kwak2007graphs}, where a more detailed treatment is available.

\begin{figure}[H]
	\centering
	\input{figGroupoid}
	\vspace*{-0.5em}
	\caption{Examples of $\cdot$ multiplication in the fundamental groupoid of $C_5$.} 
	\label{fig:groupoid}
\end{figure}
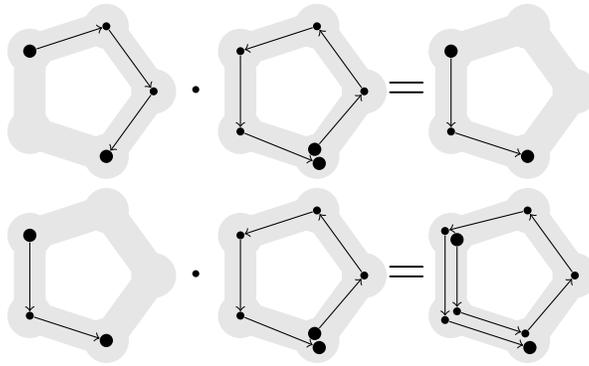

%% file: figReducing.tex
\newcommand{\dtengraph}{
	\begin{scope}[shift={(-1.3,0)}]		
		\node[H] (ha0) at (0:1.3) {};
		\node[H] (ha1) at (72:1.3) {};
		\node[H] (ha2) at (144:1.3) {};
		\node[H] (ha3) at (-144:1.3) {};		
		\node[H] (ha4) at (-72:1.3) {};	
		\draw[He]
		(ha0.center)--(ha1.center)
		(ha1.center)--(ha2.center)
		(ha2.center)--(ha3.center)
		(ha3.center)--(ha4.center)
		(ha4.center)--(ha0.center);
	\end{scope}	
	\begin{scope}[shift={(2.6,0)}]		
		\node[H] (hb0) at (180+0:1.3) {};
		\node[H] (hb1) at (180+72:1.3) {};
		\node[H] (hb2) at (180+144:1.3) {};
		\node[H] (hb3) at (180+-144:1.3) {};		
		\node[H] (hb4) at (180+-72:1.3) {};	
		\draw[He]
		(hb0.center)--(hb1.center)
		(hb1.center)--(hb2.center)
		(hb2.center)--(hb3.center)
		(hb3.center)--(hb4.center)
		(hb4.center)--(hb0.center);
	\end{scope}		
	\draw[He] (ha0.center)--(hb0.center);
}

\begin{tikzpicture}[scale=0.8]
	\dtengraph;
	
	\node[G1] (v1) at (ha2) {};
	\node[G] (v2) at (ha1) {};
	\node[G] (v3) at ($(ha0)+(-3pt,7pt)$) {};
	\node[G] (v4) at ($(hb0)+(0pt,7pt)$) {};
	\node[G] (v5) at ($(hb4)+(-72:-4pt)$) {};
	\node[G] (v6) at ($(hb3)+(-144:-4pt)$) {};
	\node[G] (v7) at ($(hb2)+(144:-4pt)$) {};
	\node[G] (v8) at ($(hb1)+(72:-4pt)$) {};
	\node[G] (v9) at ($(hb0)+(2pt,-7pt)$) {};
	\node[G] (v10) at ($(ha0)+(2pt,-5pt)$) {};
	\node[G] (v11) at ($(hb0)+(4pt,-3pt)$) {};
	\node[G] (v12) at ($(hb1)+(72:4pt)$) {};
	\node[G] (v13) at ($(hb2)+(144:4pt)$) {};
	\node[G] (v14) at ($(hb3)+(-144:4pt)$) {};
	\node[G] (v15) at ($(hb4)+(-72:4pt)$) {};
	\node[G] (v16) at ($(hb0)+(2pt,2pt)$) {};
	\node[G] (v17) at ($(ha0)+(-3pt,-0pt)$) {};
	\node[G1] (v18) at (ha4) {};					
	\foreach \x [remember=\x as \lastx (initially 1)] in {2,...,18}
	{ \draw[Ge,->] (v\lastx)--(v\x); }
	
\begin{scope}[shift={(7.5,0)}]
	\dtengraph;

	\node[G1] (v1) at (ha2) {};
	\node[G] (v2) at ($(ha1)+(0pt,-5pt)$) {};
	\node[G] (v3) at ($(ha0)+(-6pt,1pt)$) {};
	\node[G] (v4) at ($(ha1)+(2pt,4pt)$) {};
	\node[G] (v5) at ($(ha0)+(72:4pt)$) {};
	\node[G] (v6) at ($(hb0)+(0,4pt)$) {};
	\node[G] (v7) at ($(hb4)+(108:5pt)$) {};
	\node[G] (v8) at ($(hb3)+(72:4pt)$) {};
	\node[G] (v9) at ($(hb4)+(108:0pt)$) {};
	\node[G] (v10) at ($(hb0)+(4pt,-1pt)$) {};
	\node[G] (v11) at ($(hb1)+(-108:-4pt)$) {};
	\node[G] (v12) at ($(hb2)+(-36:-4pt)$) {};
	\node[G] (v13) at ($(hb3)+(-3pt,-4pt)$) {};
	\node[G] (v14) at ($(hb4)+(128:-6pt)$) {};
	\node[G] (v15) at ($(hb3)+(4pt,-2pt)$) {};
	\node[G] (v16) at ($(hb2)+(-36:4pt)$) {};
	\node[G] (v17) at ($(hb1)+(-108:4pt)$) {};
	\node[G] (v18) at ($(hb0)+(-2pt,-5pt)$) {};	
	\node[G] (v19) at ($(ha0)+(2pt,-5pt)$) {};	
	\node[G] (v20) at ($(ha4)+(-72:4pt)$) {};	
	\node[G] (v21) at ($(ha3)+(0pt,-0pt)$) {};		
	\node[G1] (v22) at ($(ha4)+(-72:-4pt)$) {};					
	\foreach \x [remember=\x as \lastx (initially 1)] in {2,...,22}
	{ \draw[Ge,->] (v\lastx)--(v\x); }	
\end{scope}

\begin{scope}[shift={(0,-3.6)}]
	\dtengraph;
	
	\node[G1] (v1) at (ha2) {};
	\node[G] (v2) at (ha1) {};
	\node[G] (v3) at (ha0) {};
	\node[G1] (v4) at (ha4) {};
	\foreach \x [remember=\x as \lastx (initially 1)] in {2,...,4}
	{ \draw[Ge,->] (v\lastx)--(v\x); }
		
\begin{scope}[shift={(7.5,0)}]
	\dtengraph;
	
	\node[G1] (v1) at (ha2) {};
	\node[G] (v2) at (ha1) {};
	\node[G] (v3) at ($(ha0)+(0pt,4pt)$) {};
	\node[G] (v4) at ($(hb0)+(0pt,4pt)$) {};
	\node[G] (v5) at (hb4) {};
	\node[G] (v6) at (hb3) {};
	\node[G] (v7) at (hb2) {};
	\node[G] (v8) at (hb1) {};
	\node[G] (v9) at ($(hb0)-(0pt,4pt)$)  {};						
	\node[G] (v10) at ($(ha0)-(0pt,4pt)$)  {};							
	\node[G1] (v11) at (ha4) {};
	\foreach \x [remember=\x as \lastx (initially 1)] in {2,...,11}
	{ \draw[Ge,->] (v\lastx)--(v\x); }
\end{scope}

\end{scope}

\draw[thick,->,decorate,decoration=snake] (0.6,-1) -- (0.6,-2.05);
\draw[thick,->,decorate,decoration=snake] (4.7,-1) -- (4.0,-2.15);
\end{tikzpicture}

%% file: figGroupoid.tex
\newcommand{\dfivegraph}{
	\node[H] (ha0) at (0:1.3) {};
	\node[H] (ha1) at (72:1.3) {};
	\node[H] (ha2) at (144:1.3) {};
	\node[H] (ha3) at (-144:1.3) {};		
	\node[H] (ha4) at (-72:1.3) {};	
	\draw[He]
	(ha0.center)--(ha1.center)
	(ha1.center)--(ha2.center)
	(ha2.center)--(ha3.center)
	(ha3.center)--(ha4.center)
	(ha4.center)--(ha0.center);
}

\begin{tikzpicture}[scale=0.7]
	\dfivegraph;
	
	\node[G1] (v1) at (ha2) {};
	\node[G] (v2) at (ha1) {};
	\node[G] (v3) at (ha0) {};
	\node[G1] (v4) at (ha4) {};
	\foreach \x [remember=\x as \lastx (initially 1)] in {2,...,4}
	{ \draw[Ge,->] (v\lastx)--(v\x); }	
	
\node at (2.1,0) {\Huge$\cdot$};	

\begin{scope}[shift={(4,0)}]
	\dfivegraph;
	
	\node[G1] (v1) at ($(ha4)-(-72:4pt)$) {};
	\node[G] (v2) at (ha0) {};
	\node[G] (v3) at (ha1) {};
	\node[G] (v4) at (ha2) {};
	\node[G] (v5) at (ha3) {};		
	\node[G1] (v6) at ($(ha4)+(-72:4pt)$) {};
	\foreach \x [remember=\x as \lastx (initially 1)] in {2,...,6}
	{ \draw[Ge,->] (v\lastx)--(v\x); }
\end{scope}

\node at (6.1,0) {\huge$=$};

\begin{scope}[shift={(8,0)}]
	\dfivegraph;
	
	\node[G1] (v1) at (ha2) {};
	\node[G] (v2) at (ha3) {};
	\node[G1] (v3) at (ha4) {};
	\foreach \x [remember=\x as \lastx (initially 1)] in {2,...,3}
	{ \draw[Ge,->] (v\lastx)--(v\x); }
\end{scope}

\begin{scope}[shift={(0,-3.5)}]

	\dfivegraph;
	
	\node[G1] (v1) at (ha2) {};
	\node[G] (v2) at (ha3) {};
	\node[G1] (v3) at (ha4) {};
	\foreach \x [remember=\x as \lastx (initially 1)] in {2,...,3}
	{ \draw[Ge,->] (v\lastx)--(v\x); }

\node at (2.1,0) {\Huge$\cdot$};

\begin{scope}[shift={(4,0)}]
	\dfivegraph;
	
	\node[G1] (v1) at ($(ha4)-(-72:4pt)$) {};
	\node[G] (v2) at (ha0) {};
	\node[G] (v3) at (ha1) {};
	\node[G] (v4) at (ha2) {};
	\node[G] (v5) at (ha3) {};		
	\node[G1] (v6) at ($(ha4)+(-72:4pt)$) {};
	\foreach \x [remember=\x as \lastx (initially 1)] in {2,...,6}
	{ \draw[Ge,->] (v\lastx)--(v\x); }
\end{scope}

\node at (6.1,0) {\huge$=$};

\begin{scope}[shift={(8,0)}]
	\dfivegraph;
	
	\node[G1] (v1) at ($(ha2)-(144:4pt)$) {};
	\node[G] (v2) at ($(ha3)-(-144:4pt)$) {};
	\node[G] (v3) at ($(ha4)-(-72:4pt)$) {};
	\node[G] (v4) at (ha0) {};
	\node[G] (v5) at (ha1) {};
	\node[G] (v6) at ($(ha2)+(144:4pt)$) {};
	\node[G] (v7) at ($(ha3)+(-144:4pt)$) {};
	\node[G1] (v8) at ($(ha4)+(-72:4pt)$) {};
	\foreach \x [remember=\x as \lastx (initially 1)] in {2,...,8}
	{ \draw[Ge,->] (v\lastx)--(v\x); }
\end{scope}	

\end{scope}	
\end{tikzpicture}

%% file: commute.tex
A \emph{square} \marginpar{\footnotesize{square}}
in a graph $G$ is a quadruple of vertices $v_1,v_2,v_3,v_4$ such that $\{v_1,v_2\}, \{v_2,v_3\}, \{v_3,v_4\},$ and $\{v_4,v_1\}$ are edges of $G$.
A square is \emph{trivial} if $v_1=v_3$ or $v_2=v_4$, \emph{non-trivial} otherwise (since we work with graphs without loops only, a square is non-trivial if and only if its vertices are pairwise different).
A graph is \emph{square-free} if it has no non-trivial squares.

The fundamental groupoid $\fundPi{G}$ turns out to be too fine-grained for two reasons.
One is that we would like $\fundPi{G\times H}$ to have something in common with $\fundPi{G}\times \fundPi{H}$: intuitively, the product of two cycles should behave like a torus, see later Figure~\ref{fig:torus}.
Another reason is recoloring: if $v_1,v_2,v_3,v_4$ is a square, a walk going through $v_1,v_2,v_3$ can be changed to go through $v_1,v_4,v_3$ instead, by changing just one value, so we want to allow such a replacement in $\fundPi{G}$ too.

Therefore, for a graph $G$, we define $\sim$ \marginpar{$\sim$}
to be the smallest equivalence relation between walks in $G$ which makes a walk $W$ equivalent to its reduction $\red{W}$, and makes $W \cat v_{1} v_{2} \cat v_{2}v_{3} \cat  W'$ equivalent to $W \cat v_{1} v_{4} \cat v_{4}v_{3} \cat W'$ for all walks $W,W'$ and every non-trivial square $v_{1}, v_2, v_{3}, v_{4}$ in $G$.
In other words, two walks are equivalent under $\sim$ if and only if one can be obtained from the other by a series of \emph{elementary steps},
\marginpar{\footnotesize{elementary\\step}}
each step consisting of either deleting or introducing a subwalk of the form $e\cat e^{-1}$ (for some edge $e$) or replacing a subwalk $v_{1} v_{2} \cat v_{2}v_{3}$ with $v_{1} v_{4} \cat v_{4}v_{3}$ (for some non-trivial square $v_{1}, v_2, v_{3}, v_{4}$ in $G$).
Note the words `non-trivial' can be dropped without changing the definition, because if $v_2=v_4$, a subwalk $v_1 v_2 \cat v_2 v_3$ is of course equal to $v_1 v_4 \cat v_4 v_3$, and if $v_1=v_3$, a subwalk $v_1 v_2 \cat v_2 v_3$ can be deleted and a subwalk $v_1 v_4 \cat v_4 v_3$ can be introduced (in other words, they both reduce to $\eps$).

We denote the equivalence class of $P$ under $\sim$ as $\sed{P}$.
Two walks equivalent under $\sim$ must have the same initial and final vertex. Since $P \sim P'$ and $Q\sim Q'$ implies $P \cat Q \sim P' \cat Q'$ for walks $P,P',Q,Q'$, and since $P \sim \red{P}$, we have that  $P \sim P'$ and $Q\sim Q'$ implies $P \cdot Q \sim P' \cdot Q'$ for reduced walks.
Also $P \sim P'$  implies $P^{-1}\sim P'^{-1}$.\marginpar{$\coarsePi{G}$,\\ $\sed{P}$}
Therefore the quotient groupoid $\coarsePi{G}$ (and group $\coarsepi{v}{G}$)
can be defined naturally on equivalence classes of walks under $\sim$.
We note a similar quotient was considered in~\cite{simons2012generalized}.

The equivalence of $\fundPiSymbol$ and $\coarsePiSymbol$ for square-free graphs follows from definitions:

\begin{lemma}\label{lem:squarefreeIso}
	Let $K$ be a square-free graph. Then the function from $\coarsePi{K}$ to $\fundPi{K}$ mapping $\sed{W}$ to $\red{W}$ for any walk $W$ in $K$ is well defined and is a groupoid isomorphism.
\end{lemma}

Quite naturally, a graph homomorphism implies a groupoid homomorphism for $\coarsePiSymbol$.

\begin{lemma}\label{lem:homo}
	Let $\mu: G\to K$ be a graph homomorphism.
	Then the function from $\coarsePi{G}$ to $\coarsePi{K}$ mapping $\sed{W}$ to $\sed{\mu(W)}$ for any walk $W$ in $G$ is well defined and is a groupoid homomorphism.
\end{lemma}
\begin{proof}
	We want to show that if $W\sim W'$ for two walks in $G$, then $\mu(W) \sim \mu(W')$. It suffices to check this when $W$ and $W'$ differ by one elementary step.
	If $W'$ is obtained from $W$ by deleting (or introducing) a subsequence $g_0 g_1 \cat g_1 g_0$ for some $g_0g_1 \in G$, then $\mu(W')$ is obtained from $\mu(W)$ by deleting (or introducing) the subsequence $\mu(g_0g_1) \cat \mu(g_1 g_0)$. Hence $\mu(W)\sim \mu(W')$.
	If $W'$ is obtained from $W$ by replacing a subsequence $g_1 g_2 \cat g_2 g_3$ with $g_1 g_4 \cat g_4 g_3$ for some square $g_1,g_2,g_3,g_4$ in $G$, then $\mu(W')$ is obtained from $\mu(W)$ by replacing the corresponding images of $\mu$. Since $\mu$ is a graph homomorphism, $\mu(g_1),\mu(g_2),\mu(g_3),\mu(g_4)$ is a square in $K$, and hence $\mu(W)\sim \mu(W')$.
	Thus $W \sim W'$ implies $\mu(W)\sim\mu(W')$, meaning the function is well defined.
	It is indeed a groupoid homomorphism, because $\mu(W^{-1})=\mu(W)^{-1}$, and $\mu(W \cdot W') \sim \red{\mu(W \cdot W')}= \red{\mu(\red{W \cat W'})} = \red{\mu(W \cat W')} = \red{\mu(W)\cat\mu(W')} = \mu(W) \cdot \mu(W')$.
\end{proof}

A crucial observation is that if $\red{W}=\red{W'}$ or more generally $W\sim W'$, then the lengths of $W$ and $W'$ have the same parity (this follows immediately by considering elementary steps).
We \marginpar{\footnotesize{length parity}}
can hence speak of the parity of an element of $\coarsePi{G}$.

We would like to think of $\coarsePi{G\times H}$ as being isomorphic to $\coarsePi{G}\times\coarsePi{H}$, but one may see that by projecting a walk in $G\times H$ to $G$ and $H$, we can never get a pair of walks of different parity.
Except for this problem (which could be resolved by taking $G\times K_2$ instead of $G$), they are in fact equivalent, and we will only need the following slightly weaker lemma.
Unfortunately the proof is quite technical; note also the lemma would not be true if we considered $\fundPi{G}$ instead of $\coarsePi{G}$.

\begin{lemma}\label{lem:product}
	Let $G,H$ be graphs.
	The function $\phi:\coarsePi{G\times H} \to \coarsePi{G} \times \coarsePi{H}$ mapping $\sed{W}$ to $\left(\sed{W|_G},\sed{W|_H}\right)$ for any walk $W$ in $G\times H$ is well defined and gives an injective groupoid homomorphism.
\end{lemma}
\begin{proof}
	We first show that if $W \sim W'$ for two walks $W,W'$ in $G\times H$, then $W|_G \sim W'|_G$ (and symmetrically $W|_H \sim W|_H'$).
	It suffices to show this when $W$ and $W'$ differ by an elementary step.
	If $W'$ is obtained from $W$ by a single reduction deleting (or introducing) a subwalk $(g_1,h_1)(g_2,h_2)\cat (g_2,h_2)(g_1,h_1)$, then $W'|_G$ is obtained from $W|_G$ by a single reduction deleting (or introducing) $g_1 g_2 \cat g_2 g_1$, and hence $W|_G \sim W'|_G$.
	If $W'$ is obtained from $W$ by a replacing a subwalk $v_1v_2 \cat v_2 v_3$ with $v_1 v_4 \cat v_4 v_3$ for some square $v_1,v_2,v_3,v_4$ in $G\times H$,
	then $W'|_G$ is obtained from replacing the corresponding subsequences after projection, where $v_1|_G,v_2|_G,v_3|_G,v_4|_G$ is a square in $G$, and hence $W|_G \sim W'|_G$.
	Thus $W \sim W'$ implies $W|_G \sim W'|_G$ and $W|_H \sim W'|_H$, so $\phi(\sed{W}) = (\sed{W|_G},\sed{W|_H})$ unambiguously defines a function from $\coarsePi{G\times H}$ to $\coarsePi{G} \times \coarsepi{h}{H}$.
	Clearly $(W\cdot W')|_G \sim \red{(W\cdot W')|_G} = \red{\red{W \cat W'}|_G} = \red{W \cat W'|_G} = \red{\red{W}|_G \cat \red{W'}|_G} = \red{W}|_G \cdot \red{W'}|_G$ and $W^{-1}|_G = W|_G^{-1}$, so $\phi$ defines a groupoid homomorphism. It remains to show that $\phi$ is injective.

	\def\join{\texttt{join}}
	For walks $P$ in $G$ and $Q$ in $H$ such that $|P|=|Q| \mod 2$, define $\join(P,Q)$ as the following walk in $G\times H$. If $|P|\geq |Q|$, let $\join(P,Q)$ be the walk whose projection to $G$ is $P$ and whose projection to $H$ is $Q \cat e^{-1} \cat e \cat \dots \cat e^{-1}\cat e$, where $e$ is an arbitrary edge ending in the same vertex as $Q$, repeated $|P|-|Q|$ times here.
	Otherwise, if $|P|<|Q|$, define $\join(P,Q)$ analogously, extending $P$ with an arbitrary edge $e$ so that it's length matches the length of $Q$,
	that is, $\join(P,Q)|_G = P \cat e^{-1} \cat e \dots \cat e^{-1} \cat e$ and $\join(P,Q)|_H = Q$.
	
	Observe that $\red{\join(P,Q)|_G} = \red{P}$, $\red{\join(P,Q)|_H} = \red{Q}$ for every pair $P,Q$ for which it is defined, and $W = \join(W|_G, W|_H)$ for every walk $W$ in $G\times H$.
	We claim that for any walks $P,P',Q,Q'$, if $P \sim P'$ and $Q\sim Q'$ then $\join(P,Q) \sim \join(P',Q')$ if both joins are defined.
	It suffices to show this in the case $P$ differs from $P'$ with an elementary step and $Q'=Q$ (we can show the case $Q$ differs from $Q'$ and $P=P'$ in a symmetric way).
	
	If $P'$ is obtained from $P$ by replacing a subwalk $g_1 g_2 \cat g_2 g_3$ with $g_1 g_4 \cat g_4 g_2$ for some square $g_1,g_2,g_3,g_4$ in $G$,
	then $|P|=|P'|$, so $\join(P',Q)$ is obtained from $\join(P,Q)$ by replacing a subwalk $(g_1,h_1) (g_2,h_2) \cat (g_2,h_2) (g_3,h_3)$ with $(g_1,h_1) (g_4,h_2) \cat (g_4,h_2) (g_3,h_3)$ for vertices $h_1,h_2,h_3 \in V(H)$ (that is, the projections to $G$ are as described and projections to $H$ are unchanged).
	Then $(g_1,h_1),(g_2,h_2),(g_3,h_3),(g_4,h_2)$ is a square in $G\times H$, so $\join(P,Q)\sim\join(P',Q)$.
	
	\smallskip
	If $P'$ is obtained from $P$ by introducing a subwalk $g_1 g_2 \cat g_2 g_1$ for some $g_1,g_2\in V(G)$,
	then let $P=P_1\cat P_2$ and $P'=P_1 \cat g_1 g_2 \cat g_2 g_1 \cat P_2$ for some walks $P_1,P_2$ in $G$.
	We prove the claim by induction on the length $P_2$.
	
	For the inductive step, let $P_2$ be non-empty, so $P_2 = g_1 g_x \cat P_3$ for some edge $g_1 g_x$ and walk $P_3$ of $G$.
	Define an intermediate walk $P''=P_1 \cat g_1 g_x \cat g_x g_1 \cat P_2$.
	Then $P'$ and $P''$ have the same lengths, so $\join(P',Q)=\join(P_1 \cat g_1 g_2 \cat g_2 g_1 \cat P_2, Q)$
	is obtained from $\join(P'',Q)$ by replacing a subwalk 
	$(g_1,h_1) (g_2,h_2) \cat (g_2,h_2) (g_1,h_3)$
	with a subwalk $(g_1,h_1) (g_x,h_2) \cat (g_x,h_2) (g_1,h_3)$ for some $h_1,h_2,h_3 \in V(H)$.
	Note $(g_1,h_1), (g_2,h_2), (g_1,h_3), (g_x,h_2)$ is a square in $G\times H$, so  $\join(P',Q) \sim \join(P'',Q)$.
	Since $P''=P_1 \cat g_1 g_x \cat g_x g_1 \cat g_1 g_x \cat P_3$ can be obtained from $P=P_1 \cat g_1 g_x \cat P_3$ by introducing $g_x g_1 \cat g_1 g_x$ before $P_3$, whose length is shorter than $P_2$, we know by inductive assumption that $\join(P,Q)\sim\join(P'',Q)$ and hence $\join(P,Q)\sim\join(P',Q)$.
	
	For the basis of the induction assume now that $P_2$ is empty.
	Suppose first that $P=P_1$ is strictly shorter than $Q$.
	Then by definition of $\join$, $\join(P,Q) = \join(P \cat e^{-1} \cat e, Q)$  for some edge $e$ with the same endpoint as $P$, that is, $e=(g_x,g_1)$ for some $g_x \in V(G)$.
	In this case $\join(P',Q)=\join(P \cat g_1 g_2 \cat g_2 g_1, Q)$ can be obtained from $\join(P,Q)$ by replacing the subwalk whose projection to $G$ is $e^{-1}\cat e$, namely $(g_1,h_1) (g_x,h_2) \cat (g_x,h_2) (g_1,h_3)$, with a walk whose projection to $G$ is the introduced fragment (and the projection to $H$ is unchanged), that is, $(g_1,h_1)(g_2,h_2) \cat (g_2,h_2) (g_1,h_3)$.
	Note $(g_1,h_1), (g_x,h_2), (g_1,h_3), (g_2,h_2)$ is a square in $G\times H$, so $\join(P,Q) \sim \join(P',Q)$.
	
	Suppose now that $P_2$ is empty and $P=P_1$ is at least as long as $Q$.
	Then $\join(P',Q)=\join(P\cat g_1 g_2 \cat g_2 g_1,Q)$ is obtained from $\join(P,Q)$ by appending $(g_1,h_1)(g_2,h_2) \cat (g_2,h_2)(g_1,h_1)$ to it, where $e=h_1h_2$ is the edge with which $Q$ would be extended in the definition of $\join$.
	Thus $\join(P,Q)\sim\join(P',Q)$.
	\smallskip
	
	This concludes the proof that for any walks $P,P',Q,Q'$, if $P \sim P'$ and $Q\sim Q'$, then $\join(P,Q) \sim \join(P',Q')$, if both joins are defined.
	Thus we can unambiguously define the function $\join'(\sed{P},\sed{Q}) := \sed{\join(P,Q)}$ for walks $P,Q$ whose lengths have the same parity. Since $\join'(\phi(\sed{W})) = \sed{\join(W|_G,W|_H)} = \sed{W}$, the function $\join'$ is the inverse of $\phi$.
	Thus $\phi$ is an injection.
	In fact, $\phi$ gives an isomorphism between $\coarsepi{(g,h)}{G\times H}$ and the subgroup of $\coarsepi{g}{G}\times \coarsepi{h}{H}$ formed by those pairs $(\sed{P},\sed{Q})$ where $|P|$ and $|Q|$ have the same parity.
\end{proof}

\pagebreak
The only cases for which we will use the above `product lemma' are the next two corollaries, focusing on closed walks of $G$ and $H$.
For \marginpar{$\lcycle{C}{h_0 h_1}$}
a closed walk $C$ in a graph $G$ and an oriented edge $h_0h_1$ in a graph $H$, we define $\lcycle{C}{h_0 h_1}$ as the closed walk in $G\times H$ whose projection to $G$ is $C\cat C$ and whose projection to $H$ is $h_0 h_1 \cat h_1 h_0$ repeated $|C|$ times.
For a closed walk $D$ in $H$ and $g_0g_1\in G$ we define $\rcycle{g_0 g_1}{D}$ symmetrically.
Elements of the form $\sed{\mu(\lcycle{C}{h_0 h_1})}$ in $\coarsePi{K}$ will be our central tool.
The `product lemma' then translates to the following.

\begin{corollary}\label{cor:factor}
	Let $\mu: G\times H \to K$.
	Let $g_0 g_1 \in G$ and $h_0 h_1\in H$.	
	Let $C\in\fundpi{(g_0,h_0)}{G\times H}$.
	Then $\pexp{\sed{\mu(C)}}{2}=\sed{\mu(\lcycle{{C|_G}}{h_0 h_1})} \cdot \sed{\mu(\rcycle{g_0 g_1}{C|_H})}$.
\end{corollary}
\begin{proof}
	Since $h_0 h_1 \cat h_1 h_0 \sim \eps$ in $H$ and $\phi$ is a groupoid homomorphism by Lemma~\ref{lem:product}, we have
	\begin{gather*}
	\phi(\sed{C^2}) = (\sed{C^2|_G}, \sed{C^2|_H}) = (\sed{C^2|_G}, \sed{\eps}) \cdot (\sed{\eps}, \sed{C^2|_H}) =\\
	= (\sed{{C|_G}^2}, \sed{h_0 h_1 \cat h_1 h_0 \cat \dots}) \cdot (\sed{g_0 g_1 \cat g_1 g_0 \cat \dots}, \sed{{C|_H}^2}) =\\
	= \phi(\sed{\lcycle{C|_G}{h_0 h_1}}) \cdot \phi(\sed{\rcycle{g_0 g_1}{C|_H}}) = \phi(\sed{\lcycle{C|_G}{h_0 h_1}} \cdot \sed{\rcycle{g_0 g_1}{C|_H}}).
	\end{gather*}
	Lemma~\ref{lem:product} also says that $\phi$ is injective, and hence $\pexp{\sed{C}}{2} = \sed{\lcycle{C|_G}{h_0 h_1}} \cdot \sed{\rcycle{g_0 g_1}{C|_H}}$.
	By Lemma~\ref{lem:homo}, $\pexp{\sed{\mu(C)}}{2} = \sed{\mu(\lcycle{{C|_G}}{h_0 h_1})} \cdot \sed{\mu(\rcycle{g_0 g_1}{C|_H})}$.
\end{proof}

The crucial property of the product is that if $X$ is a closed walk of the form $\lcycle{C}{h_0 h_1}$ and $Y$ is of the form $\rcycle{g_0 g_1}{D}$, then $X$ and $Y$ commute in $\coarsePi{G\times H}$, that is $X \cdot Y = Y \cdot X$.
Therefore, by Lemma~\ref{lem:homo} $\sed{\mu(X)}$ commutes with $\sed{\mu(Y)}$, which will gives us a lot of information (in the case of square-free $K$).

\begin{corollary}\label{cor:commute}
	Let $\mu: G\times H \to K$.
	Let $g_0 g_1 \in G$ and $h_0 h_1\in H$.
	Let $C \in  \fundpi{g_0}{G}$ and $D \in \fundpi{h_0}{H}$.
	Then $\sed{\col(\lcycle{C}{h_0 h_1})}$ commutes with $\sed{\col(\rcycle{g_0 g_1}{D})}$.	
\end{corollary}
\begin{proof}
	Since $h_0h_1 \cat h_1 h_0 \sim \eps$ in $H$,
	$\phi(\sed{\lcycle{C}{h_0 h_1}})=(\sed{C^2|_G},\sed{\eps})$ and similarly
	$\phi(\sed{\rcycle{g_0 g_1}{D}})=(\sed{\eps},\sed{D^2|_H})$.
	As $\sed{\eps}$ commutes with any element of $\coarsepi{g_0}{G}$ (and similarly for $H$), $(\sed{C^2|_G},\sed{\eps})$ commutes with $(\sed{\eps},\sed{D^2|_H})$.
	By Lemma~\ref{lem:product}, $\phi$ is an injective homomorphism and hence $\sed{\lcycle{C}{h_0 h_1}}$ commutes with $\sed{\rcycle{g_0 g_1}{D}}$.
	Since $\mu$ is a graph homomorphism, by Lemma~\ref{lem:homo}, $\sed{\mu(\lcycle{C}{h_0 h_1})}$ commutes with $\sed{\mu(\rcycle{g_0 g_1}{D})}$.
\end{proof}

\begin{figure}[H]
	\centering
	\input{figTorus}
	\caption{The graph $C_6 \times C_7$. For $h\in V(H)$, the vertices $(g,h)$ are drawn in two columns depending on the parity of $g+h$ to make the structure of squares more apparent.
	The red line shows $\rcycle{0\,1}{C_7}$.
	The blue line shows $\lcycle{C_6}{0\,1}$ (visiting each edge twice, since $C_6$ is even).} 
	\label{fig:torus}
\end{figure}
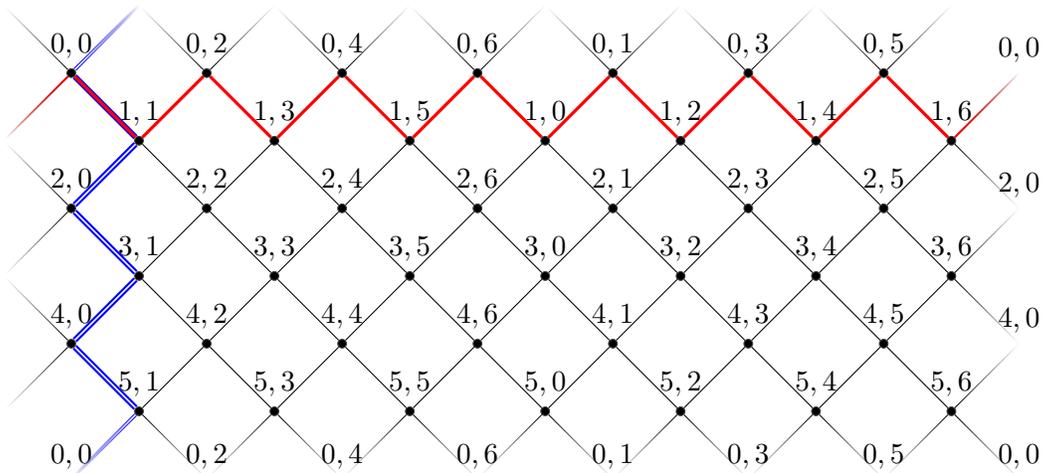
\pagebreak

We now consider the information given by $\coarsePi{K}$ more precisely.
For the graphs $K$ we consider, $\coarsepi{v}{K}$ is a free group, for any $v \in V(K)$.
For square-free graphs this follows from the fact that $\fundpi{v}{K}$ is always a free group (see e.g.~\cite{kwak2007graphs}).
For circular cliques this follows from the fact that $\coarsepi{v}{K_{p/q}}$ is isomorphic to $\ZZ$, for $2<\frac{p}{q}<4$ (Lemma~\ref{lem:cycleGroup}).
The property of free groups we need (and which can easily be checked directly for $\fundpi{v}{K}$ and $\ZZ$) is that primitive roots can be unambiguously defined and that primitive roots of commuting elements are equal, up to inversion.
For \marginpar{\footnotesize{primitive\\root}}
an element $O$ of a free group $\grouppi$ other than the trivial element $\eps$, its \emph{primitive root} is the unique $R\in \grouppi$ such that $O=R^n$ for some $n\in\mathbb{N}$ with $n$ maximized (see e.g.~\cite{madlener1980string} for a linear time algorithm computing $R$).

\begin{fact}\label{fact:commuteRoot}
	Let $O_1,O_2$ be elements of a free group.
	Then $O_1$ and $O_2$ commute, i.e. $O_1\cdot O_2=O_2\cdot O_1$, if and only if
	$O_1=\varepsilon$ or $O_2=\varepsilon$ or their primitive roots are equal or the inverse of each other. 
\end{fact}

Turning our attention to square-free graphs $K$, from Corollary~\ref{cor:commute} (and Lemma~\ref{lem:squarefreeIso}) we have that $\red{\col(\lcycle{C}{h_0 h_1})}$ commutes with $\red{\col(\rcycle{g_0 g_1}{D})}$ for any cycles $C,D$ in $G,H$, and therefore they have the same primitive root (up to inversion), if they are both non-$\eps$ elements of $\fundpi{\mu(g_0,h_0)}{K}$.
Intuitively, this means that if we take any cycle in $G$ and any cycle in $H$, then $\mu$ maps them to closed walks that wind around the same cycles (or the same sequence of cycles) in $K$, though they may wind a different number of times and in opposite directions.

Since this is true for any pair of cycles, this implies that either all cycles in $G$ map to $\eps$, or all cycles in $H$ map to $\eps$, or all cycles in $G$ and $H$ map to closed walks winding around one common cycle of $K$.
We make this more formal in the following proof.
The theorem captures all we need from this section for the case of square-free $K$.
In the first and second case we will later be able to directly obtain a graph homomorphism from $G$ and $H$, respectively, while in the third case, we will reduce our problem by obtaining a homomorphism $G\times H \to C_n$, where $C_n \to K$ corresponds to the common primitive root.

\begin{theorem}\label{thm:commute}
	Let $\mu: G\times H \to K$ for a square-free graph $K$.
	Let $g_0 g_1 \in G, h_0 h_1 \in H$.
	Then one of the following holds:
	\begin{itemize}
		\item $\red{\mu(C)}=\eps$ for every closed walk $C$ from $(g_0,h_0)$ in $G\times h_0 h_1$,
		\item $\red{\mu(D)}=\eps$ for every closed walk $D$ from $(g_0,h_0)$ in $g_0 g_1 \times H$,
		\item there is an $R\in \fundPi{K}$ such that for every closed walk $C'$ from $(g_0,h_0)$ in $G\times H$, $\red{\mu(C')}=\pexp{R}{i}$ for some $i\in\ZZ$.
	\end{itemize}
\end{theorem}
\begin{proof}
	Suppose first that $\red{\mu(\lcycle{C}{h_0 h_1})} = \eps$ for all $C\in \fundpi{g_0}{G}$.
	Let $C'$ be any closed walk from $(g_0,h_0)$ in $G\times h_0 h_1$.
	Then $C'\cat C' = \lcycle{C'|_G}{h_0 h_1}$, and hence $\pexp{\red{\mu(C')}}{2}=\red{\mu(\lcycle{\red{C'|_G}}{h_0 h_1})} = \eps$, implying $\red{\mu(C')}=\eps$. So the first case of our claim holds.
	Symmetrically, if  $\red{\mu(\rcycle{g_0 g_1}{D})} = \eps$ for all $D\in \fundpi{h_0}{H}$, then the second case of our claim holds.
	
	If neither of the above two possibilities holds, then there is a $C_0\in\fundpi{g_0}{G}$ with $\red{\mu(\lcycle{C_0}{h_0 h_1})}\neq \eps$ and a $D_0\in\fundpi{h_0}{H}$ with $\red{\mu(\rcycle{g_0 g_1}{D_0})}\neq \eps$.
	Let $R$ be the primitive root of $\red{\mu(\lcycle{C_0}{h_0 h_1})}$, $R\neq\eps$.
	
	Let $D$ be any element of $\fundpi{h_0}{H}$.
	By Corollary~\ref{cor:commute} and Lemma~\ref{lem:squarefreeIso}, $\red{\mu(\lcycle{C_0}{h_0 h_1})}$ commutes with $\red{\mu(\rcycle{g_0 g_1}{D})}$.
	By Fact~\ref{fact:commuteRoot}, this implies that the primitive root of $\red{\mu(\rcycle{g_0 g_1}{D_0})}$ is $R$, up to inversion.
	That is, for every $D\in \fundpi{h_0}{H}$, $\red{\mu(\rcycle{g_0 g_1}{D})}=\pexp{R}{i}$ for some $i\in\ZZ$. 
	Using $D_0$, we can symmetrically show that for every $C\in \fundpi{g_0}{G}$, $\red{\mu(\lcycle{C}{h_0 h_1})}=\pexp{R}{i}$ for some $i\in\ZZ$.
	
	For any $C'\in\fundpi{(g_0,h_0)}{G\times H}$, by
	Corollary~\ref{cor:factor} (and Lemma~\ref{lem:squarefreeIso}), $\pexp{\red{\mu(C')}}{2}=\red{\mu(\lcycle{C'|_G}{h_0 h_1})} \cdot \red{\mu(\rcycle{g_0 g_1}{C'|_H})}$
	and hence $\pexp{\red{\mu(C')}}{2}=\pexp{R}{i}$ for some $i\in\ZZ$.
	So either $\red{\mu(C')}$ is empty or it has the same primitive root as  $\pexp{\red{\mu(C')}}{2}$, and in both cases $\red{\mu(C')}=\pexp{R}{i}$ for some $i\in\ZZ$.
\end{proof}

%% file: figTorus.tex
\begin{tikzpicture}[scale=0.9]
	\tikzstyle{ghost}=[inner sep=0pt,minimum size=0pt]
	\tikzstyle{v}=[circle,fill=black,draw=black!75,inner sep=0pt,minimum size=0.3em]
	\tikzstyle{splitl}=[path fading=west]
	\tikzstyle{splitr}=[path fading=east]

\pgfmathtruncatemacro\n{6}
\pgfmathtruncatemacro\m{7}

\pgfmathtruncatemacro\mm{2*\m}

\foreach \i in {-1,...,\n}
{
	\foreach \j in {-1,...,\mm}
	{
		\pgfmathtruncatemacro\t{mod(mod(\i+\n,\n)+mod(\j+\mm,\mm),2)};
		\pgfmathtruncatemacro\jj{mod(\j,\m)};
		\pgfmathtruncatemacro\ii{mod(\i,\n)};		
		\ifnum \t < 1
			\ifnum \i < 0
				\node[ghost] (u\i_\j) at (\j,-\i) {};
			\else \ifnum \j < 0
				\node[ghost] (u\i_\j) at (\j,-\i) {};
			\else \ifnum \i = \n
				\node[ghost,label={above:$\ii,\jj$}] (u\i_\j) at (\j,-\i) {};
			\else \ifnum \j = \mm
				\node[ghost,label={above:$\ii,\jj$}] (u\i_\j) at (\j,-\i) {};
			\else
				\node[v,label={above:$\ii,\jj$}] (u\i_\j) at (\j,-\i) {};
			\fi \fi \fi \fi
		\fi
	}
}

\foreach \i in {-1,...,\n}
{
	\foreach \j in {-1,...,\mm}
	{
		\pgfmathtruncatemacro\t{mod(mod(\i+\n,\n)+mod(\j+\mm,\mm),2)};
		\pgfmathtruncatemacro\ip{\i+1};
		\pgfmathtruncatemacro\im{\i-1};		
		\pgfmathtruncatemacro\jp{\j+1};	
				
		\ifnum \t < 1
			\ifnum \i < 0
				\ifnum \jp < \mm
					\draw[splitl] (u\i_\j) -- (u\ip_\jp);			
				\fi
			\else \ifnum \i = \n
				{};
			\else \ifnum \j = \mm
				{};
			\else \ifnum \j < 0
				\ifnum \ip < \n
					\draw[splitl] (u\i_\j) -- (u\ip_\jp);
				\fi
			\else \ifnum \ip = \n
				\draw[splitr] (u\i_\j) -- (u\ip_\jp);
			\else \ifnum \jp = \mm
				\draw[splitr] (u\i_\j) -- (u\ip_\jp);			
			\else
				\draw (u\i_\j) -- (u\ip_\jp);
			\fi \fi \fi \fi \fi \fi
			\ifnum \i < 0
				{};
			\else \ifnum \i = 0
				\ifnum \j < \mm
					\draw[splitr] (u\i_\j) -- (u\im_\jp);
				\fi
			\else \ifnum \j = \mm
				{};
			\else \ifnum \j < 0
				\draw[splitl] (u\i_\j) -- (u\im_\jp);			
			\else \ifnum \jp = \mm
				\draw[splitr] (u\i_\j) -- (u\im_\jp);
			\else \ifnum \i = \n
				\draw[splitl] (u\i_\j) -- (u\im_\jp);				
			\else
				\draw (u\i_\j) -- (u\im_\jp);
			\fi \fi \fi \fi \fi \fi

			\ifnum \j = 0 \ifnum \i < \n
			\draw[blue,thick,double] (u\i_\j) -- (u\ip_\jp);
			\fi \fi						
			\ifnum \j = 0 \ifnum \i < \n \ifnum \i>0
			\draw[blue,thick,double] (u\i_\j) -- (u\im_\jp);
			\fi \fi \fi
			\ifnum \j = 0 \ifnum \i = 0
			\draw[blue,double,splitr] (u\i_\j) -- (u\im_\jp);
			\fi \fi 
			\ifnum \j = 0 \ifnum \i = \n
			\draw[blue,double,splitl] (u\i_\j) -- (u\im_\jp);
			\fi \fi	
			
			\ifnum \i = 0 \ifnum \j < \mm 
				\draw[red,very thick] (u\i_\j) -- (u\ip_\jp);
			\fi \fi			
			\ifnum \i = 1 \ifnum \jp < \mm \ifnum \j>-1
				\draw[red,very thick] (u\i_\j) -- (u\im_\jp);
			\fi \fi \fi
			\ifnum \i = 1 \ifnum \j = -1
			\draw[red,thick,splitl] (u\i_\j) -- (u\im_\jp);
			\fi \fi 
			\ifnum \i = 1 \ifnum \jp = \mm
				\draw[red,thick,splitr] (u\i_\j) -- (u\im_\jp);
			\fi \fi

		\fi
	}
}
\end{tikzpicture}

%% file: circular.tex
We begin this section by showing that circular cliques with $2<\frac{p}{q}<4$ behave like circles, topologically, and so the (coarse) fundamental group just describes an integer: the winding number.
(An analogous fact is well known for the box complex: formally, it is homotopy equivalent to a circle and its fundamental group is isomorphic to $\mathbb{Z}$ for circular cliques with $2<\frac{p}{q}<4$).
For simplicity, we only consider odd $p$; this includes the case of odd cycles in particular (as $C_{2n+1}$ is isomorphic to $K_{2n+1/n}$), and will still allow us to conclude the general case.

\begin{lemma}\label{lem:cycleGroup}
	Let $p,q$ be integers such that $2<\frac{p}{q}<4$ and $p$ is odd.
	Then $\coarsepi{v}{K_{p/q}}$ is a group isomorphic to $\ZZ$, for any $v\in V(K_{p/q})$.
\end{lemma}
\begin{proof}
	We begin by defining a more intuitive view of a circular clique (so that edges will join numbers that are close enough, instead of far enough).
	We need the following definitions:
	\begin{itemize}
		\item For $i,j\in \ZZ_{2p}$, define $\od(i-j)$ to be the integer (in $\ZZ$) in the set $\{-(p-1),-(p-2),\dots,p-1,p\}$ which is equivalent to $i-j \mod 2p$
		(this depends only on $i-j$, but we think of it as a signed distance between $i$ and $j$ in the circle $\ZZ_{2p}$).
		\item Let $K'$ be the graph with $V(K')=\{0,1,\dots,2p-1\}$, $E(K')=\{ij \mid \od(i-j)\mbox{ is odd and }|\od(i-j)|\leq p-2q\}$ (that is, $K'$ is the Cayley graph of $\ZZ_{2p}$ with generators $\pm1,\pm3,\dots,\pm p-2q$).
		\item Define $\phi: K_{p/q} \times K_2 \to K'$ as $\phi(i,0)=2i \mod 2p$ and $\phi(i,1) = 2i + p \mod 2p$. This is easily seen to be a graph isomorphism.
		Indeed, vertices whose difference is 
		$(i-j)\in \{q,q+1,\dots,\lfloor\frac{p}{2}\rfloor,\lceil\frac{p}{2}\rceil,\dots,p-q\} \mod p$, will map to vertices whose difference is\\
		$2(i-j)+p \in \{2q+p,2q+2+p,\dots, 2\lfloor\frac{p}{2}\rfloor+p, 2\lceil\frac{p}{2}\rceil+p,\dots, 2(p-q)+p\} =$\\
		\hspace*{5em}$=\{-(p-2q),-(p-2q-2),\dots,-1,1,\dots,p-2q\} \mod 2p$.
		\item For a walk $W$ in $K_{p/q}$ whose $(i+1)$-th edge is $w_i w_{i+1}$, define $\varphi(W)$ to be the walk in $K'$ of the same length whose $(i+1)$-th edge is $\phi(w_i, i\mod 2) \phi(w_i, i+1\mod 2)$.
		Note that $\varphi$ maps closed walks of odd length beginning and ending in $w_0\in V(K_{p/q})$ to walks between $\phi(w_0,0)$ and $\phi(w_0,1)=\phi(w_0,0)+p \mod 2p$, which are not closed.
		\item For a walk $W'$ in $K'$ whose $(i+1)$-th edge is $w_i w_{i+1}$, define $\pd(W)= \sum_{i=0}^{|W|-1} \od(w_{i+1}-w_i)$.
		Intuitively, $\pd$ measures how far $W$ went winding around $K'$, and we should think of a closed walk $W$ as winding $d$ times if $\pd(W)=d \cdot 2p$.
	\end{itemize}		
	
	We claim that $\pd$ defines a functor from $\coarsePi{K'}$ to the group $\ZZ$,
	that is, $\pd(W \cat W') = \pd(W) + \pd(W')$ and $W\sim W'$ implies $\pd(W)=\pd(W')$ for any two walks $W,W'$ in $K'$.
	The first is clear, so consider two walks $W\sim W'$ in $K'$.
	It suffices to show that $\pd(W)=\pd(W')$ when $W$ and $W'$ differ by an elementary step.
	If $W'$ is obtained from $W$ by introducing/deleting a subsequence $w_1 w_2 \cat w_2 w_1$ for some $w_1w_2 \in E(K')$, then $\pd(W')$ is obtained from $\pd(W)$ by adding/subtracting $\pd(w_1 w_2 \cat w_2 w_1) = \od(w_2 - w_1) + \od(w_1 - w_2) = 0$, so the two values are indeed equal (note that $\od(w_2 - w_1) = -\od(w_1 - w_2)$ unless both are equal to $p$, which is impossible when $w_1w_2 \in E(K')$).
	If $W'$ is obtained from $W$ by replacing a subsequence $ab \cat bc$ by $ad \cat dc$, for some square $a,b,c,d$ in $K'$, then $\pd(W')$ differs from $\pd(W)$ by $(\od(b-a)+\od(c-b))-(\od(d-a)+\od(c-d)) = \od(b-a) + \od(a-d) + \od(d-c) + \od(c-b)$.
	Since $ab$ is an edge of $K'$, we have $|\od(b-a)|\leq p-2q$ and similarly $|\od(a-d)|\leq p-2q$.
	Hence $|\od(b-a)|+|\od(a-d)| \leq 2p-4q < p$, so $a,b,d$ are contained in an interval of length less than $p$ in $\ZZ_{2p}$, implying $\od(b-a)+\od(a-d)=\od(b-d)$ and $|\od(b-d)|<p$.
	Similarly $\od(d-c) + \od(c-b) = \od(d-b)$.
	Therefore the difference between $\pd(W')$ and $\pd(W)$ is $\od(b-d)+\od(d-b)=0$, so they are in fact equal.
	
	Furthermore, $\pd \circ \varphi$ is a functor from $\coarsePi{K_{p/q}}$ to the group $\ZZ$, that is, $\pd(\varphi(W\cat W'))=\pd(\varphi(W))+\pd(\varphi(W'))$ and $W\sim W'$ implies $\pd(\varphi(W))=\pd(\varphi(W'))$.
	The first follows from the definitions and the fact that $\od((p+i)-(p+j)) = \od(i-j)$ for $i,j \in \ZZ_{2p}$.
	The second follows from the fact that an elementary step showing $W \sim W'$ in $K_{p/q}$ corresponds to an elementary step showing $\phi(W) \sim \phi(W')$ in $K'$; in particular, if $a,b,c,d$ is a square in $K_{p/q}$, then $\phi(a,i),\phi(b,1-i),\phi(c,i),\phi(d,1-i)$ is a square in $K'$ for $i=0,1$.
	
	Let us define a generator for $\coarsepi{0}{K_{p/q}}$.
	Define $O$ as the closed walk of length $p$ in $K_{p/q}$ whose $(i+1)$-th edge is $(i \cdot \lceil\frac{p}{2}\rceil \mod p, (i+1)\cdot \lceil\frac{p}{2}\rceil \mod p)$.
	Then $\varphi(O)$ is a walk of length $p$ in $K'$ going from $0$ to $p$, whose $(i+1)$-th edge is $(i \mod 2p, i+1 \mod 2p)$; indeed, for even $i$, it is by definition $(2\cdot i \lceil\frac{p}{2}\rceil \mod 2p, 2\cdot (i+1) \lceil\frac{p}{2}\rceil +p \mod 2p)=(i\cdot(p+1) \mod 2p, (i+1)\cdot(p+1)+p \mod 2p)=(i \mod 2p, i+1 \mod 2p)$, and similarly for odd $i$.
	Thus $\pd(\varphi(O)) = \sum_{i=0}^{|O|} 1 = |O| = p$.
	
	We claim that for every closed walk $W$ in $K_{p/q}$ from $0$ to $0$, there is a $d\in\ZZ$ such that $W\sim O^d$.
	First, we use elementary steps to transform $W$ so that $w_{i+1} - w_{i} \in \{\lfloor\frac{p}{2}\rfloor, \lceil\frac{p}{2}\rceil\}$ (as elements in $\ZZ_p$) for any edge $w_i w_{i+1}$ of $W$.
	Indeed, if say $w_{i+1} - w_{i} \in  \{q,q+1,\dots,\lfloor\frac{p}{2}\rfloor-1\}$, then letting $x=w_i + \lfloor\frac{p}{2}\rfloor$ and $y=w_{i}-1$,
	we see that $w_i -x =  \lfloor\frac{p}{2}\rfloor$, $x-y=\lceil\frac{p}{2}\rceil$, and $w_{i+1}-y = w_{i+1} -w_{i} +1 \mod p $.
	In particular, $w_i, x,y,w_{i+1}$ is a square in $K_{p/q}$, so $w_i w_{i+1} \sim w_i w_{i+1} \cat w_{i+1} y \cat y w_{i+1} \sim w_i x \cat x y \cat y w_{i+1}$.
	Hence we can replace the subwalk $w_i w_{i+1}$ in $W$ by $w_i x \cat x y \cat y w_{i+1}$.
	Since this introduced two edges with difference between endpoints in $\{\lfloor\frac{p}{2}\rfloor, \lceil\frac{p}{2}\rceil\}$ and changed this difference for the third edge to be closer to $\lfloor \frac{p}{2}\rfloor$, we can do such replacements until we get a walk $W' \sim W$ with $w'_{i+1} - w'_{i} \in \{\lfloor\frac{p}{2}\rfloor, \lceil\frac{p}{2}\rceil\}$ for any edge $w'_i w'_{i+1}$ of $W'$.
	Then, if two consecutive edges have a different difference, say, $w'_{i+1}-w'_{i} = \lfloor\frac{p}{2}\rfloor$ and $w'_{i+2}-w'_{i+1}=\lceil\frac{p}{2}\rceil$, then in fact $w'_{i+2} = w'_{i} + \lfloor\frac{p}{2}\rfloor + \lceil\frac{p}{2}\rceil = w'_i$, so they reduce, that is, the subwalk $w'_i w'_{i+1} \cat w'_{i+1} w'_{i+2}$ can be deleted in an elementary step.
	We do this until we get a walk $W''\sim W$ such that $w''_{i+1}-w''_{i} = c$ for all edges $w''_{i+1} w''_{i}$ of $W''$, for some constant $c\in \{\lfloor\frac{p}{2}\rfloor,\lceil\frac{p}{2}\rceil\}=\{\lfloor\frac{p}{2}\rfloor,-\lfloor\frac{p}{2}\rfloor\}$.
	Then the $i$-th vertex of the walk is $w''_i = i \cdot c$.
	Since $W''$ is a closed walk, it must be that $|W''| \cdot c = 0 \mod p$.
	Therefore, $p$ must divide $|W''|$ and $W'' = O^{|W|/p}$ or $W'' = O^{-|W|/p}$.
	
	The above paragraph shows that every element of $\coarsepi{0}{K_{p/q}}$ is of the form $\sed{O^d}$ for some $d\in\ZZ$.
	Since $\Delta(\varphi(O^d)) = d \cdot p$, these are pairwise different elements, for different $d$.
	Clearly $\sed{O^d} \cdot \sed{O^{d'}} = \sed{O^{d+d'}}$, hence $\coarsepi{0}{K_{p/q}}$ is a group isomorphic to $\ZZ$.
	For any $v\in V(K_{p/q})$, let $P$ be any walk from $0$ to $v$ in $K_{p/q}$.
	Then $\sed{O} \mapsto \sed{P} \cdot \sed{O} \cdot \sed{P}^{-1}$ is easily checked to be a group isomorphism between $\coarsepi{v}{K_{p/q}}$ and $\coarsepi{0}{K_{p/q}}$.
\end{proof}

Next, with give a very short proof of a parity argument used in~\cite{El-ZaharS85,HaggkvistHMN88,DelhommeS02}.
For \marginpar{\footnotesize{half-parity}}
a cycle $C$ in $G$ or $H$, if $\sed{\mu(\lcycle{C}{h_0 h_1})}=\pexp{X}{2}$ for some $X\in\coarsePi{K}$, define the \emph{half-parity} of $C$ as the parity of $|X|$.
The following lemma shows that the half-parity of each odd-length cycle in $G$ is defined and different from the half-parity of each odd-length cycle in $H$.

\begin{lemma}\label{lem:paritySwitch}
	Let $\mu:G\times H\to K$, $g_0g_1\in G$, $h_0h_1\in H$, and assume $\coarsepi{\mu(g_0,h_0)}{K}$ is a free~group.
	Let $C$ be a closed walk from $g_0$ in $G$ and
	let $D$ be a closed walk from $h_0$ in $H$, with $|C|$,$|D|$~odd.
	Then $\sed{\mu(\lcycle{C}{h_0 h_1})} = \pexp{R}{2i}$ and
	$\sed{\mu(\rcycle{g_0 g_1}{D})} = \pexp{R}{2j}$ for some $R\in \coarsepi{\mu(g_0,h_0)}{K}$ of odd length and $i,j \in \ZZ$ such that $i+j$ is odd.
\end{lemma}	
\begin{proof}
	Let $J$ be the closed walk from $(g_0,h_0)$ in $G\times H$ whose projection to $G$ is $C^{|D|}$ and whose projection to $H$ is $D^{|C|}$.
	Then $J$ has length $|C|\cdot |D|$, which is odd, in particular $\sed{\mu(J)} \neq \eps$.
	Let $R$ be the primitive root of $\sed{\mu(J)}$,
	that is, $\sed{\mu(J)} =\pexp{R}{k}$, for some $k\in \ZZ$.
	It follows that $R$ and $k$ are odd.
	By Corollary~\ref{cor:factor},
	$$
	\pexp{R}{2k}=\pexp{\sed{\mu(J)}}{2} =\sed{\mu(\lcycle{(C^{|D|})}{h_0 h_1})} \cdot \sed{\mu(\rcycle{g_0 g_1}{(D^{|C|})})}=\pexp{\sed{\mu(\lcycle{C}{h_0 h_1})}}{|D|}\ \cdot\ \pexp{\sed{\mu(\rcycle{g_0 g_1}{D})}}{|C|}
	$$

	By Corollary~\ref{cor:commute}, $\sed{\mu(\lcycle{C}{h_0 h_1})}$ and $\sed{\mu(\rcycle{g_0 g_1}{D})}$ commute.
	Hence they both commute with $\pexp{R}{2k}$ and by Fact~\ref{fact:commuteRoot},
	they are equal to $\pexp{R}{2i}$ and $\pexp{R}{2j}$ respectively, 
	for some $i,j\in\ZZ$
	(the exponents must be even because $R$ is odd and $|\lcycle{C}{h_0 h_1}|$ is even).
	Then $\pexp{R}{2k}= \pexp{R}{(2i |D| + 2j |C|)}$, so
	$i \cdot |D| + j \cdot |C| \equiv i+j \mod 2$ must be odd.
\end{proof}

This implies that either the half-parity is odd for all odd-length cycles in $G$ and even for all odd-length cycles in $H$, or vice versa.
(As a side note, let us mention this conclusion could be reached more generally, even when $\coarsePi{K}$ is not free, as long as the edges of $K$ admit an orientation such that no square $a,b,c,d$ of $K$ is oriented $a \to b \to c \to d$ and $a\to d$; if such an orientation exists, the algebraic length mod $4$ of a walk turns out to be a suitable invariant.)

\smallskip
The key to the parity approach is however in the next lemma, and in the corollary following it.
It will show that if an odd-length cycle in $G$ has odd half-parity, then $\mu(g,h_0)$ is equal or close to $\mu(g,h_1)$ for some vertex $g$ of the cycle.
Excluding such a vertex (or edge) from every odd cycle, we will later make a large subgraph of $G$ bipartite.

Intuitively, the lemma reflects the topological fact that in a map from a circle to a circle $\mu:S^1 \to S^1$ winding an odd number of times, there must be a pair of antipodal points that maps to antipodal points, that is, a point $x\in S^1\subseteq \mathbb{R}^2$ satisfying $\mu(x)=-\mu(-x)$.
The idea is then that given a cycle $C$ in $G$, we can view $\mu$ as a map from $\lcycle{C}{h_0 h_1}$ to $K \times K_2$, which can be extended piece-wise linearly to a continuous map from a circle to the topological space corresponding to $K\times K_2$.
If $C$ has odd half-parity, then this map will be winding an odd number of times.
The above fact then implies that some antipodal points map to antipodes in $K\times K_2$ and hence the to the same point in $K$.
If $K$ is an odd cycle and $|C|$ is odd, it can be shown that such antipodal points will occur as vertices of $\lcycle{C}{h_0 h_1}$ (instead of some general position in the continuous extension).
This is not true for circular cliques, but we can still show a slight relaxation.

\begin{lemma}\label{lem:fixpoint}
	Let $O=k_0k_1 \cat k_1 k_2 \cat \dots \cat k_{2n-1} k_0$ be a closed walk of length $2n$ in $K_{p/q}$, for $n,p$ odd and $2<\frac{p}{q}<4$.
	If $\sed{O} = \sed{R}^{\cdot 2}$ for some walk $R$ of odd length in $K_{p/q}$, then there is an index $i\in \ZZ_{2n}$ such that $k_i k_{i+n+1}$ and $k_{i+1} k_{i+n}$ are edges of $K_{p/q}$.
\end{lemma}
\begin{proof}
	We reuse the definitions of $\od,K',\pd,\varphi$ of the proof of Lemma~\ref{lem:cycleGroup}.
	Let us first translate the statement in these terms.
	In particular, $\varphi(O)$ is a closed walk in $K'$ of length $2n$.
	Since $|R|$ is odd, $\pd(\varphi(R))$ is odd too (from the definitions, it is a sum of $|R|$ summands, each of which corresponds to an edge of $K'$ and hence is an odd integer).
	We showed that $O \sim R^2$ implies $\pd(\varphi(O))=\pd(\varphi(R^2))=2\cdot \pd(\varphi(R))$
	and hence $\pd(\varphi(O)) \equiv 2 \mod 4$,
	which is all we need to know about $\varphi(O)$.

	Let $\varphi(O)=c_0 c_1 \cat c_1 c_2 \cat \dots \cat c_{2n-1} c_0$.
	We wish to show that for some $i\in \ZZ_{2n}$, $k_i k_{i+n+1}$ and $k_{i+1} k_{i+n}$ are edges of $K_{p/q}$. This is the same as saying that the difference between endpoints is in $\{q,q+1,\dots,\lfloor\frac{p}{2}\rfloor,\lceil\frac{p}{2}\rceil,\dots,p-q\} \mod p$,
	which by definition of $\varphi$ is equivalent to saying that the difference between endpoints of $c_i c_{i+n+1}$ and $c_{i+1} c_{i+n}$ is in
	$$\{2q,2q+2,\dots,2\lfloor\frac{p}{2}\rfloor,2\lceil\frac{p}{2}\rceil,\dots,2(p-q)\} = 
	\{2q,2q+2,\dots,p-1,-(p-1),\dots,-2q\},$$
	that is,
	$|\od(c_i - c_{i+n+1})|\geq 2q$ and $|\od(c_{i+1} - c_{i+n})|\geq 2q$.
	We can forget $K_{p/q}$ and focus on the walk $\varphi(O)$ in $K'$ from now on.
	
	The statement we want to prove is now the following:
	if $\varphi(O)=c_0 c_1 \cat c_1 c_2 \cat \dots \cat c_{2n-1} c_0$ is a walk in $K'$ of length $2n$ for $n$ odd such that $\pd(\varphi(O)) \equiv 2 \mod 4$, then there is an $i\in\ZZ_{2n}$ such that $|\od(c_i - c_{i+n+1})|\geq 2q$ and $|\od(c_{i+1} - c_{i+n})|\geq 2q$.
	
	Suppose to the contrary that for all $i\in\ZZ_{2n}$, $|\od(c_i - c_{i+n+1})|< 2q$ or $|\od(c_{i+1} - c_{i+n})|< 2q$.
	We claim that for all $i\in\ZZ_{2n}$, 
	\begin{equation}\label{eq:fixpointproof}\tag{*}
	\od(c_{i+1}-c_i) - \od(c_{i+n+1}-c_{i+n}) = \od(c_{i+1}-c_{i+n+1}) - \od(c_i - c_{i+n})
	\end{equation}
	Fix $i\in \ZZ_{2n}$ and assume first that $|\od(c_i - c_{i+n+1})|< 2q$. Then $c_i c_{i+1} \in E(K')$ means $|\od(c_{i+1} - c_{i})|\leq p-2q$ and hence $c_i,c_{i+1}$ and $c_{i+n+1}$ are contained in an interval of length less than $p$ of $\ZZ_{2p}$, which implies $\od(c_{i+1} - c_{i})+\od(c_i - c_{i+n+1})=\od(c_{i+1} - c_{i+n+1})$.
	Similarly $c_{i+n+1}c_{i+n} \in E(K')$ implies that $\od(c_i - c_{i+n+1}) + \od(c_{i+n+1} - c_{i+n}) = \od(c_i - c_{i+n})$.
	Subtracting the two gives~\eqref{eq:fixpointproof}.
	Note also that $|\od(c_i - c_{i+n})|<2q + p-2q = p$.
	The proof is analogous in the other case, when $|\od(c_{i+1} - c_{i+n})|< 2q$.
	
	Let us now sum \eqref{eq:fixpointproof} over $i=0,1,\dots,n-1$.
	The left side then amounts to $\sum_{i=0}^{n-1} \od(c_{i+1}-c_i) - \sum_{i=n}^{2n-1} \od(c_{i+1}-c_i)$, while the right side telescopes to simply $-\od(c_0-c_{0+n})+\od(c_{n-1+1}-c_{n-1+n+1}) = -\od(c_0-c_{n})+\od(c_{n}-c_{0}) = 2 \od(c_{n}-c_{0})$ (the last equality follows from $|\od(c_0-c_{n})|<p$).
	Since $n$ is odd, $c_{n}$ and $c_{0}$ belong to different sides of the bipartition of $K'$ and hence $\od(c_{n}-c_{0})$ is odd.
	Therefore  $\sum_{i=0}^{n-1} \od(c_{i+1}-c_i) - \sum_{i=n}^{2n-1} \od(c_{i+1}-c_i) = 2 \od(c_{n}-c_{0}) \equiv 2 \mod 4$.
	Since similarly $\od(c_{i+1}-c_i)$ is odd for all $i$, and $n$ is odd, we have $2 \cdot \sum_{i=n}^{2n-1} \od(c_{i+1}-c_i) \equiv 2 \mod 4$.
	Together, this implies $\sum_{i=0}^{2n-1} \od(c_{i+1}-c_i) \equiv 0 \mod 4$.	
	But this contradicts our assumption that $\Delta(\varphi(O)) = \sum_{i=0}^{2n-1} \od(c_{i+1}-c_i)  \equiv 2 \mod 4$.
\end{proof}

\begin{corollary}\label{cor:fixpoint}
	Let $\mu:G\times H \to K_{p/q}$ for $2<\frac{p}{q}<4$ and $p$ odd. Let $h_0h_1\in H$.
	Let $C$ be an odd-length closed walk in $G$.
	If $C$ has odd half-parity,
	then there is an edge $gg'$ of $C$ such that $\mu(g,h_0)\mu(g',h_0)$ and $\mu(g,h_1)\mu(g',h_1)$ are edges in $K_{p/q}$.
\end{corollary}
\begin{proof}
	Recall $C$ having odd half-parity means $\sed{\mu(\lcycle{C}{h_0 h_1})} = \pexp{X}{2}$ for some odd $X \in \coarsePi{K_{p/q}}$.
	The claim follows then from Lemma~\ref{lem:fixpoint} applied to the closed walk $\mu(\lcycle{C}{h_0 h_1})$: it has length $2|C|$, where $|C|$ is odd, and 
	vertices indexed with $i$, $i+|C|+1$, $i+1$ and $i+|C|$ are $\mu(g,h_j), \mu(g',h_j), \mu(g',h_{1-j})$ and $\mu(g,h_{1-j})$ respectively, for some $gg' \in C$ and $j\in\{0,1\}$.
\end{proof}	

Finally, we use what we obtained to get a graph homomorphism $G\to K$, similarly as in~\cite{El-ZaharS85}, except for using the relaxed condition on edges instead of a condition on vertices.

\begin{lemma}\label{lem:cutOddCycles}
	Let $\mu: G\times H \to K$.  Let $g_0g_1\in G$, $h_0h_1\in H$.	
	If every odd-length closed walk in $G$ has an edge $gg'$ such that $\mu(g,h_0)\mu(g',h_0)\in K$ and $\mu(g,h_1)\mu(g',h_1)\in K$, then $G\to K$.
\end{lemma}
\begin{proof}
	Let $G'$ be the subgraph of $G$ obtained by removing those edges $gg'\in G$ which satisfy $\mu(g,h_0)\mu(g',h_0) \in K$ and $\mu(g,h_1)\mu(g',h_1) \in K$.
	Then the assumption says that $G'$ is bipartite.
	Fix a bipartition of $G'$ and let $\delta(g)=h_0$ for $g\in V(G)$ on one side of it and $\delta(g)=h_1$ for $g$ on the other side. In other words, $\delta$ is a graph homomorphism from $G'$ to $h_0 h_1$, a subgraph of $H$ isomorphic to $K_2$.
	
	Define $\gamma: V(G) \to V(K)$ as $\gamma(g) = \mu(g,\delta(g))$ for $g\in V(G)$.
	To show that $\gamma$ is a graph homomorphism, consider any edge $gg'\in G$.
	If $gg' \in G'$, then $\delta(gg')\in H$ (in fact $\delta(gg')=h_0 h_1$) and $gg' \in G$, which implies $\gamma(gg')=\mu((g,\delta(g))(g',\delta(g'))) \in K$.
	If $gg' \not\in G'$, then either $\delta(gg')\in H$ and $\gamma(gg')\in E(K)$ follows as before, or $\delta(g)=\delta(g')=h_i$ for some $i\in\{0,1\}$, which implies $\gamma(gg')=\mu(g,h_i)\mu(g',h_i)$, which is an edge of $K$ by construction of $G'$.
\end{proof}

Since the above lemma is the one that gives the final graph homomorphism, we note a potentially interesting generalization which follows straightforwardly from the same proof: let $\mu:G\times H\to K$, let $H'$ be an induced subgraph of $H$, and let $G'$ be the subgraph of $G$ obtained by removing those edges $gg' \in E(G)$ such that $\forall_{h,h' \in V(H')} \mu(g,h)\mu(g',h')\in K$; then $G'\to H'$ implies $G\to K$.

\begin{theorem}\label{thm:circular}
	The circular clique $K_{p/q}$ is multiplicative, for $2\leq \frac{p}{q} < 4$.
\end{theorem}
\begin{proof}
	We first show the claim for $p$ odd (in particular $2<\frac{p}{q}$).
	Let $\mu : G\times H \to K_{p/q}$ for some graphs $G,H$ and let $g_0g_1 \in G$, $h_0 h_1\in H$.
	We can assume $G$ and $H$ are connected and non-bipartite.
	By Lemma~\ref{lem:cycleGroup}, $\coarsepi{\mu(g_0,h_0)}{K_{p/q}}$ is isomorphic to $\ZZ$ and hence a free group.
	By Lemma~\ref{lem:paritySwitch}, for any odd-length closed walks $C,D$ in $G,H$ from $g_0,h_0$ respectively, the half-parities of $C$ and $D$ are different.
	Assume without loss of generality that the half-parity is odd for all odd-length closed walks $C$ from $g_0$ in $G$ (otherwise swap $G$ and $H$).
	
	If $C'$ is an odd-length closed walk from $g'$ in $G$, we claim $C'$ has odd half-parity too.
	Indeed, taking any even-length walk $W$ from $g_0$ to $g'$,
	$W \cat C' \cat W^{-1}$ is an odd-length closed walk from $g_0$ in $G$.
	It hence has odd half-parity, meaning $\sed{\mu(\lcycle{(W \cat C' \cat W^{-1})}{h_0 h_1})} = \pexp{X}{2}$ for some odd $X \in \coarsePi{K_{p/q}}$.
	Thus $\sed{\mu(\lcycle{C'}{h_0 h_1})} =  \pexp{Y}{2}$ for $Y= \sed{\mu(W')}^{-1} \cdot X \cdot \sed{\mu(W')}$, where $W'$ is the walk with $W'|_G=W$ and $W'|_H = (h_0 h_1 \cat h_1 h_0)^{|W|/2}$.
	Hence $C'$ has odd half-parity too.
	
	Therefore by Corollary~\ref{cor:fixpoint} every odd-length closed walk in $G$ has an edge with the property from the claim and hence Lemma~\ref{lem:cutOddCycles} gives a homomorphism $G\to K_{p/q}$.
	
	Consider now $K_{p/q}$ with $p$ even.
	Suppose $G\times H \to K_{p/q}$.
	Then $G\times H \to K_{p'/q'}$ for any $p',q'$ with $\frac{p}{q}<\frac{p'}{q'}$ and
	thus $G \to K_{p'/q'}$ or $H\to K_{p'/q'}$ for $p'$ odd.
	Since the set of rationals $2<\frac{p'}{q'}<4$ with $p'$ odd is dense in the interval $(2,4)$,
	and since $\chi_c(G) = \inf\{\frac{p'}{q'} : G \to K_{p'/q'}\}$ is known to be attained~\cite{Zhu01}, it follows that $G\to K_{p/q}$ or $H\to K_{p/q}$.
\end{proof}

%% file: recoloring.tex
As sketched in the introduction, the proof will rely on inductively improving a $K$-coloring $\mu$ of $G\times H$ by recoloring. 
Recall \marginpar{\footnotesize{recoloring}}
that we say a $K$-coloring $\mu$ of a graph $G$ can be \emph{recolored} to $\mu^*$ if there is a sequence $\mu_0,\dots,\mu_n$ of $K$-colorings of $G$ with $\mu_0=\mu, \mu_n=\mu^*$, where $\mu_{i+1}$ differs from $\mu_i$ for at most one value $g\in V(G)$.
Note that if $\mu^*$ is obtained from $\mu$ by changing colors at some independent set of vertices (a set $S\subseteq V(G)$ such that $S\times S \cap E(G) = \emptyset$), then $\mu^*$ can be obtained by recoloring (considering vertices of $S$ one by one, in any order).
Recoloring can be thought as a discrete homotopy, it preserves the topological invariants we defined before; we will need this only in the following case (see~\cite{Wrochna15} for a more constructive statement; note also this works for general $K$ by taking $\coarsePi{K}$ instead of $\fundPi{K}$).

\begin{lemma}\label{lem:invariant}
	Let $\mu,\mu^*: G \to K$ for $K$ square-free. Assume $\mu$ can be recolored to $\mu^*$.
	Let $C$ be any closed walk in $G$.
	Then $\red{\mu(C)}$ and $\red{\mu^*(C)}$ are conjugate, that is, there is a $Q\in \fundPi{K}$ such that $\red{\mu(C)} = Q \cdot \red{\mu^*(C)} \cdot Q^{-1}$.
\end{lemma}
\begin{proof}
	It suffices to prove the lemma in the case $\mu^*$ is obtained in a single step, changing the color of $g\in V(G)$ only.
	Let $C=c_0c_1 \cat c_1 c_2 \cat \dots \cat c_{n-1} c_0$.
	For any $i \in \ZZ_n$ such that $c_i=g$, since $G$ is loop-free, $c_{i-1}$ and $c_{i+1}$ are different from $g$. Thus $\mu(c_{i-1})=\mu^*(c_{i-1})$ and $\mu(c_{i+1})=\mu^*(c_{i+1})$, which means $\mu(c_{i-1}), \mu(c_{i}), \mu(c_{i+1}), \mu^*(c_{i})$ is a square in $K$.
	Since $K$ is square-free, this implies $\mu(c_{i-1})=\mu(c_{i+1})$ and thus $\red{\mu(c_{i-1} c_i \cat c_i c_{i+1})} = \eps = \red{\mu^*(c_{i-1} c_i \cat c_i c_{i+1})}$.
	Hence, if $c_0\neq g$, $\red{\mu(C)} = \red{\mu^*(C)}$,
	while if $c_0=g$, then 
	$$\red{\mu(C)} = \red{\mu(c_0 c_1)} \cdot \red{\mu(c_1c_2 \cat \dots \cat c_{n-2} c_{n-1})} \cdot \red{\mu(c_{n-1}c_{0})} = \red{\mu(c_0 c_1)} \cdot \red{\mu^*(c_1c_2 \cat \dots \cat c_{n-2} c_{n-1})} \cdot \red{\mu(c_{n-1}c_{0})} =$$
	$$= \red{\mu(c_0 c_1)} \cdot \red{\mu^*(c_0 c_1)}^{-1} \cdot  \red{\mu^*(C)} \cdot \red{\mu^*(c_{n-1} c_0)}^{-1} \cdot \red{\mu(c_{n-1}c_{0})} = Q \cdot \red{\mu^*(C)} \cdot Q^{-1}$$ for $Q=ab \cat bc$
	where $a=\mu(c_0)$, $b=\mu(c_1)=\mu^*(c_1)=\mu(c_{n-1})=\mu^*(c_{n-1})$ and $c= \mu^*(c_0)$.
\end{proof}

The above lemma, together with the observation that $Q \cdot \pexp{R}{i} \cdot Q^{-1} = \pexp{(Q \cdot R \cdot Q^{-1})}{i}$, implies that if any case of Theorem~\ref{thm:commute} is true for $\mu$, then it is also true for any $K$-coloring reachable from it by recoloring.
We use this to improve a given $K$-coloring without losing the conclusions of Theorem~\ref{thm:commute}.\looseness=-1

\medskip
By \marginpar{\footnotesize{$H$-improve}}
\emph{$H$-improving} a $K$-coloring $\mu$ of $G\times H$, we mean recoloring $\mu$ to make $\mu(\cdot,h)$ as constant as possible, for every $h\in V(H)$.
Formally, $\mu^*:G\times H \to K$ $H$-improves over $\mu$ if the number of triples $g,g'\in V(G), h\in V(H)$ such that $g,g'$ have a common neighbor in $G$ and $\mu^*(g,h)\neq \mu^*(g',h)$ is lower than for $\mu$.
We say $\mu$ can be $H$-improved by recoloring if there is a $\mu^*$ to which it can be recolored and which $H$-improves over $\mu$.

For readers familiar with covering spaces in topology, the intuitions behind `improving' can be explained in the following terms (which in fact could be made formal using the theory of graph coverings presented in~\cite{kwak2007graphs}).
Consider any base vertex $(g_0,h_0)$ of $G\times H$ with any edge $h_0h_1\in H$.
In the first case of Theorem~\ref{thm:commute}, when all cycles in $G\times h_0 h_1$ map to closed walks in $K$ that are topologically trivial, we can lift the $K$-coloring $\mu$ to a graph homomorphism mapping $G\times h_0h_1$ to the universal cover of $K$, which is a tree (its nodes are the reduced walks based at $\mu(g_0,h_0)$).
This graph homomorphism to a tree can then be folded until it becomes a homomorphism to an edge, constant on $V(G)\times \{h_0\}$ and on $V(G)\times \{h_1\}$.
We fold it by finding extremal vertices in $G\times h_0 h_1$---those which map the furthest from a fixed base vertex in the universal cover---and changing the mapping so that they map closer.
For example, if $\mu$ maps a walk starting at the base point to $ab \cat bc \cat cd \cat dc \cat cb \cat ba$, then we recolor the extremal vertex (colored $d$) so that the walk maps to $ab \cat bc \cat cb \cat bc \cat cb \cat ba$; then we recolor the new extremal vertices (colored $c$) to reach $ab \cat ba \cat ab \cat ba \cat ab \cat ba$.

We proceed similarly in the last case of Theorem~\ref{thm:commute}, when all cycles of $G\times H$ map to closed walks winding around the same root $R$ of $K$.
Instead of the universal cover, we can only lift to a covering space whose fundamental group is (instead of the trivial group) the subgroup of $\fundPi{K}$ generated by $R$. In other words, we measure for each vertex, using any walk from the base vertex to it, how far this walk (as mapped in $K$) goes outside $R$.
Folding vertices extremal in this sense, we eventually reach a $K$-coloring that maps all such walks within $R$, which means there is a graph homomorphism to a cycle which admits a homomorphism to $K$.
Using the multiplicativity of cycles concludes the proof.

\medskip
Formally, \marginpar{\footnotesize{$H$-extremal}}
for $\mu:G\times H \to K$, an \emph{$H$-extremal set} is a pair $(S,h_0h_1)$ where $h_0h_1 \in H$ and $S$ is a subset of $V(G)\times\{h_1\}$ that is monochromatic, whose neighborhood is monochromatic, and whose second neighborhood is non-empty, with colors different from the color of $S$ (Figure~\ref{fig:extremal}, left).
That is, $\col(S)=\{a\}$, $\col(N_{G\times h_0 h_1} (S))=\{b\}$, $N^2_{G\times h_0 h_1} (S)\neq \emptyset$ and $a \not\in\col(N^2_{G\times h_0 h_1} (S))$, for some $a,b\in V(K)$.

The following technical lemma gives our basic inductive argument.
Intuitively, if we find an $H$-extremal set $(S,h_0h_1)$, then we can $H$-improve $\mu$ by recoloring $S$ to match some color in its second neighborhood.
If this is not immediately possible, because the colors would conflict with some $\mu(\cdot,h_2)$, then by square-freeness we will find that the conflicting values give a smaller $H$-extremal set.

\vfill
\begin{figure}[H]
	\centering
	\input{figExtremal}
	\caption{Illustration for the proof of Lemma~\ref{lem:recoloring}: $K$-colored vertices of $G\times H$ (arranged in columns according to their $H$ coordinate). Left: an $H$-extremal set (dark red) $S$, its neighborhood (blue) and second neighborhood in $G\times h_0 h_1$.
	Middle: $S'$ cannot have a neighbor outside $S$.
	Right: $S'$ has a non-empty second neighborhood.}
	\label{fig:extremal}
\end{figure}
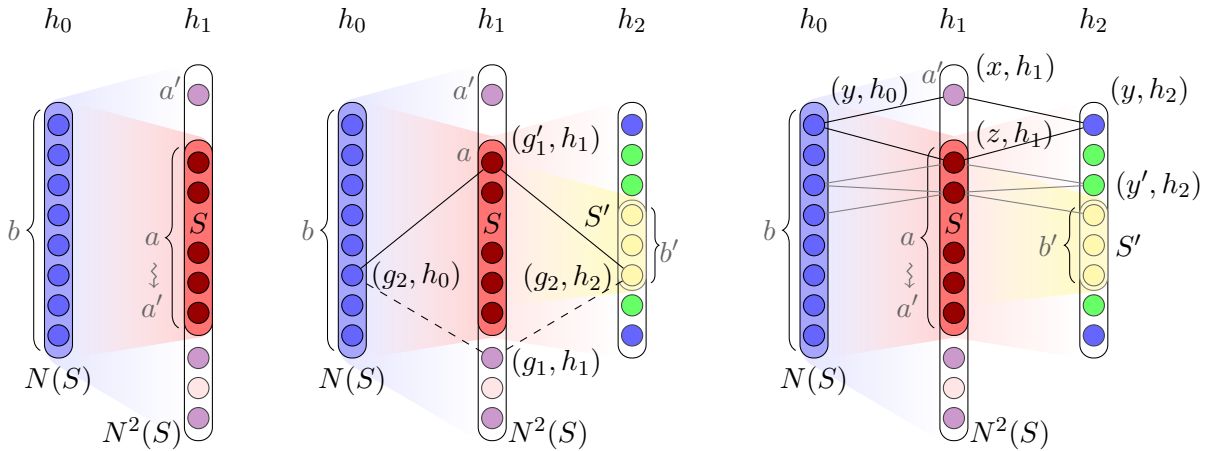
\vfill

\pagebreak

\begin{lemma}\label{lem:recoloring}
	Let $\mu:G\times H\to K$ for $K$ square-free, $G$ connected and non-bipartite.
	If there is an $H$-extremal set, then $\mu$ can be $H$-improved by recoloring.
\end{lemma}
\begin{proof}
Choose an $H$-extremal set $(S,h_0h_1)$ minimizing $|S|+|N_{G\times h_0 h_1} (S)|$.
Let $a'$ be any color in $\col(N^2_{G\times h_0 h_1} (S))$, $a'\neq a$.
Consider recoloring $S$ from $a$ to $a'$, that is, consider the assignment $\col^*:V(G\times H) \to V(K)$ obtained from $\col$ by setting $\col^*(g,h_1) := a'$ for $(g,h_1)\in S$.
It is easy to see that $\mu^*$ $H$-improves over $\mu$ (indeed, the relation $\mu(g,h) \stackrel{?}{=} \mu(g',h)$ could change only for pairs with $(g,h)\in S, (g',h)\not\in S$; so $h=h_1$ and $\mu(g,h_1)=a\neq \mu(g',h_1)$, a non-equality, could only change to an equality, namely $\mu^*(g,h_1)=a'=\mu^*(g',h_1)$, which indeed happened for at least one $g'$).

If $\mu^*$ is a $K$-coloring then we are done, so assume otherwise. There must be some $(g,h_2) \in N_{G\times H}(S)$ with a color $b':= \mu^*(g,h_2)$ such that $a'b'\not\in K$.
Since $H$ is loop-free, $h_2 \neq h_1$, hence $\mu^*(g,h_2)=\mu(g,h_2)=b'$.
By definition of $H$-extremal, $N_{G\times h_0h_1}(S)$ is mapped to one color, say $b$. It must be that $a'b \in K$ (as $a'$ appears on $N^2_{G\times h_0 h_1} (S)$ and $\mu$ was a $K$-coloring) and thus $b\neq b'$ and in particular $h_2\neq h_0$.

Let $S' = N_{G\times h_2 h_1}(S) \cap \mu^{-1}(\{b'\})$. By the above, $S'$ is non-empty.
We want to show $(S',h_1h_2)$ should have been chosen instead of $(S,h_0h_1)$.

We claim that $N_{G\times h_2 h_1}(S') \subseteq S$.
Suppose to the contrary that $(g_1,h_1)(g_2,h_2) \in G\times H$ for some $(g_2,h_2) \in S'$ and $(g_1,h_1)\not \in S$.
By definition of $S'$, $(g_2,h_2)$ also has a neighbor $(g_1',h_1)\in S$.
Consider now $(g_2,h_0)$---it must be a neighbor of $(g_1,h_1)$ and $(g_1',h_1)$ as well.
Hence $(g_1,h_1)$ is in $N^2_{G\times h_0 h_1} (S)$, implying $a'':=\mu(g_1,h_1)\neq a$.
But then $\mu(g_2,h_0)=b$, $\mu(g_1',h_1)=a$, $\mu(g_2,h_2)=b'$, and $\mu(g_1,h_1)=a''$, which gives a square in $K$ with $b \neq b'$, $a\neq a''$, a contradiction.

Hence $N_{G\times h_2 h_1}(S') \subseteq S$.
Thus, we have a non-empty set $S'\subseteq G\times\{h_2\}$ such that $\mu(S')=\{b'\}$, $\mu(N_{G\times h_2h_1}(S')) = \{a\}$, and $\mu(N^2_{G\times h_2 h_1} (S')) \subseteq \mu(N_{G\times h_2 h_1}(S)\setminus S') \not\ni b'$ by choice of $S'$.

To show that $N^2_{G\times h_2 h_1} (S') \neq \emptyset$, let $(x,h_1)\in N^2_{G\times h_0h_1}(S)$, let $(y,h_0)$ be its neighbor in $N_{G\times h_0h_1}(S)$, and let $(z,h_1)$ be a neighbor of $(y,h_0)$ in $S$.
Then $(y,h_2)$ is also a neighbor of $(x,h_1)$ and $(z,h_1)$.
Since $\mu(x,h_1)\neq a = \mu(z,h_1)$, it must be that $\mu(y,h_2)=\mu(y,h_0)$ (by square-freeness of $K$) and hence $\mu(y,h_2)=b\neq b'$.
The set $S \cup N_{G\times h_0 h_1}(S)$ must be connected in $G\times h_0 h_1$, otherwise we could limit $S$ to one of the connected components at the beginning.
Thus $S \cup N_{G\times h_2 h_1}(S)$ is connected in $G\times h_2 h_1$ as well, which means it contains a path from $S'$ to $(y,h_2)$. The first vertex $(y',h_2)$ on this path such that $\mu(y',h_2)\neq b'$ then exists and is in $N^2_{G\times h_2 h_1} (S')$, showing its non-emptiness.
Hence $(S',h_1h_2)$ is an $H$-extremal set.

It remains to show that $|S'|+|N_{G\times h_2 h_1}(S')| < |S|+|N_{G\times h_0 h_1}(S)|$.
We have already proved $N_{G\times h_2 h_1}(S') \subseteq S$, so  $|N_{G\times h_2 h_1}(S')| \leq |S|$.
The inclusion $S'\subseteq N_{G\times h_2 h_1}(S)$ is strict because of $(y,h_2)$, hence $|S'| < |N_{G\times h_2 h_1}(S)| = |N_{G\times h_0 h_1}(S)|$.
Adding the inequalities gives the claim, so $(S',h_1h_2)$ indeed should have been chosen in place of $(S,h_0h_1)$ at the beginning.
\end{proof}

A $K$-coloring that cannot be improved further has no $H$-extremal sets, which we use in the following lemma to strengthen the outcomes of Theorem~\ref{thm:commute}.
For the first and second outcome we will use $H'=h_0h_1$ instead of $H$, for the third outcome we apply this lemma directly.

\begin{lemma}\label{lem:recolored}
Let $\mu:G\times H \to K$ for $K$ square-free, $G,H$ connected, and $G$ non-bipartite.
Suppose $\mu$ has no $H$-extremal sets and suppose there is an $R\in\fundpi{\mu(g_0,h_0)}{K}$ such that for every closed walk $C$ from $(g_0,h_0)$ in $G\times H$, $\red{\mu(C)} =\pexp{R}{i}$ for some $i\in\ZZ$.
Then either:
\begin{itemize}
	\item $\mu$ is constant on $V(G)\times\{h\}$ for some $h\in V(H)$, or
	\item 
	for every walk $W$ in $G\times H$ starting at $(g_0,h_0)$, $\red{\mu(W)}$ is a prefix of $\pexp{R}{i}$ for some $i\in\ZZ$.
\end{itemize}
\end{lemma}
\begin{proof}
\def\ext{\texttt{ext}}
\def\pre{\texttt{pre}}
For a reduced walk $W$ in $G\times H$ starting at $(g_0,h_0)$, define $\pre(W)$ as the longest prefix of $\red{\mu(W)}$ which is also a prefix of $\pexp{R}{i}$ for some $i\in\ZZ$ (note it might be a prefix of $R$ and $R^{-1}$ at the same time).
Define $\ext(W)$ as the remaining suffix: $\red{\mu(W)} = \pre(W) \cat \ext(W)$.

We claim for two walks $W,W'$ with same endpoints, $\ext(W)=\ext(W')$.
Indeed, since $W\cat W'^{-1}$ is a closed walk starting and ending in $(g_0,h_0)$, we have  $\pexp{R}{i} = \red{\mu(W)} \cdot \red{\mu(W')}^{-1} = \red{ \pre(W) \cat \ext(W) \cat \ext(W')^{-1} \cat \pre(W')^{-1}}$ for some $i\in\ZZ$.
All edges in $\ext(W)$ (and possibly more) must be reduced by $\ext(W')^{-1}$ in the above expression, as otherwise $\pre(W) \cat e$ would be a prefix of the reduced expression, where $e$ is the first edge of $\ext(W)$---this is impossible because $\pre(W) \cat e$ is not a prefix of $\pexp{R}{i}$ (by definition of $\pre$).
Hence $\ext(W)$ is a suffix of $\ext(W')$.
Symmetrically, $\ext(W')$ is a suffix of $\ext(W)$, so the two are equal.

Therefore, we can unambiguously define $\ext(v)$ for $v\in G\times H$ as $\ext(W)$ for any walk $W$ from $(g_0,h_0)$ to $v$.
If $\ext(v)=\eps$ for all $v\in V(G\times H)$, then the second case of the claim holds. 

\medskip
Assume then that $\ext(v)$ is not always $\eps$.
Choose $(g^*,h^*)\in V(G\times H)$ maximizing $|\ext((g^*,h^*))|$.
Let $\ext((g^*,h^*))= a_0a_1 \cat a_1a_2 \cat \dots \cat a_{n-1}a_n$ for $a_i \in V(K)$, where $n\geq 1$ by assumption.
Let $S$ be the set of vertices $s$ in $V(G)\times \{h^*\}$ with $\ext(s)=\ext((g^*,h^*))$.
As $\ext(s)$ is a walk ending at $\mu(s)$, this implies $\mu(S)=\{a_n\}$ (see Figure~\ref{fig:extension}).

We claim that $\mu(N_{G\times H}(S)) = \{a_{n-1}\}$.
Indeed, let $x \in V(G\times H)$ be a neighbor of some $s\in S$.
Let $W$ be a walk from $(g_0,h_0)$ to $s$.
Since $\red{\mu(W)} = \pre(W) \cat \ext(s)$,  
we have $\red{\mu(W\cat sx)} = \red{\pre(W) \cat \ext(s) \cat \mu(s)\mu(x)}$.
Since $W\cat sx$ is a walk to $x$ and $|\ext(x)|\leq |\ext(s)|$, the last edge of $\ext(s)$ must reduce with $\mu(s)\mu(x)$. This implies $\mu(x) = a_{n-1}$ as claimed.

Let $h'$ be any neighbor of $h^*$ in $H$.
Now either $N^2_{G\times h^* h'}(S)$ is empty or not.
In the first case, by connectedness of $G\times h^* h'$ this means that $S$ and $N_{G\times h^* h'}(S)$ cover all of $G\times h^* h'$. Since $S\subseteq V(G)\times\{h^*\}$, $S$ must be equal to the side $V(G)\times\{h^*\}$ of the bipartition of $G\times h^* h'$, and $N_{G\times h^* h'}(S)$ must be equal to the other. 
As $\mu$ is constant on $S$, the first case of the claim holds.

In the second case, if $N^2_{G\times h^* h'}(S)$ is not empty, we show that $a_n \not \in \mu(N^2_{G\times h^* h'}(S))$.
Suppose to the contrary $\mu(y) = a_n$ for some $y \in N^2_{G\times h^* h'}(S)$.
Let $x\in N_{G\times h^* h'}(S)$ be a neighbor of $y$ and let $s\in S$ be a neighbor of $x$.
As argued before, $\mu(s)=a_n$, $\mu(x)=a_{n-1}$.
Since $\mu(y)=a_n$ too, for any walk $W$ from $(g_0,h_0)$ to $x$ we have $\red{\mu(W \cat xy)} = \red{\mu(W \cat xs)}$ and hence $\ext(y)=\ext(s)$.
As $y$ is on the same side of the bipartition of $G\times h^* h'$ as $s$, this means it must have been in $S$ (by choice of $S$), a contradiction.
Thus in fact $a_n \not \in \mu(N^2_{G\times h_0 h_1}(S))$, so $(S,h_0h_1)$ would be an $H$-extremal set, meaning it is never the case that $N^2_{G\times h_0 h_1}(S)$ is not empty.
\end{proof}

\begin{figure}[H]
	\centering
	\input{figExtension}
	\caption{The images in $K$ of two walks from $(g_0,h_0)$ to $(g^*,h^*)$. Their final vertex, extending out of $R$, defines an $H$-extremal set: it is mapped to $a_n$ (the red color), it's neighbors are mapped to $a_{n-1}$ (the blue color), while second neighbors are not red.
	We could hence improve the mapping by moving $(g^*,h^*)$ to the violet color, say.}
	\label{fig:extension}
\end{figure}
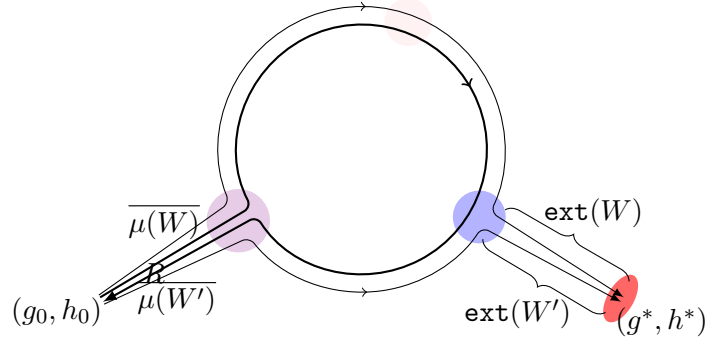

The first outcome of Lemma~\ref{lem:recolored} is easily strengthened, giving a homomorphism $H\to K$:
\pagebreak

\begin{lemma}\label{lem:recoloredConstant}
	Let $\mu:G\times H\to K$ for $K$ square-free, $G,H$ connected, and $G$ non-bipartite.
	If~$\mu$~has no $H$-extremal sets and is constant on $V(G)\times\{h\}$ for some $h\in V(H)$, then $\mu=\gamma\circ\delta$, where $\delta:G\times H \to H$ is the projection to $H$ and $\gamma: H\to K$ is a graph homomorphism.
\end{lemma}
\begin{proof}
	We first show that $\mu$ is constant on $V(G) \times \{h\}$ for every $h\in V(H)$.
	Suppose the contrary holds.
	Then by connectivity of $H$ there is an edge $h_0 h_1\in E(H)$ such that $\mu$ is constant on $V(G)\times\{h_0\}$ and is not constant on $V(G)\times\{h_1\}$.
	Let $a \in \mu(V(G)\times\{h_1\})$ and let $S = \mu^{-1}(a)\cap (V(G)\times\{h_1\})$.
	Then by connectivity of $G\times h_0 h_1$ it is easy to see that $(S,h_0h_1)$ is an $H$-extremal set, contradicting the assumption on $\mu$.
	
	Thus we can define $\gamma: V(H) \to V(K)$ by letting $\gamma(h)$ be the unique value in $\mu(V(G)\times \{h\})$.
	Clearly $\mu=\gamma\circ\delta$, where $\delta:G\times H \to H$ is the projection to $H$ and $\gamma:H \to K$.
\end{proof}

For\marginpar{\footnotesize{cyclically\\reduced}}
the other outcome of Lemma~\ref{lem:recolored}, we first need to show that $R$ is not only reduced, but is \emph{cyclically reduced}, meaning $R \cat R$ is reduced. 
This follows easily by temporarily considering a different base point.

\begin{lemma}\label{lem:cyclicallyReducedR}
	Let $\mu:G\times H\to K$ for $K$ square-free, $G$ and $H$ connected and non-bipartite.
	Suppose $\mu$ has no $H$-extremal sets.
	Suppose $R\in\fundpi{\mu(g_0,h_0)}{K}$ is such that for every closed walk $C$ from $(g_0,h_0)$ in $G\times H$, $\red{\mu(C)} =\pexp{R}{i}$ for some $i\in\ZZ$.
	If $\mu$ is not constant on $V(G)\times\{h\}$ for any $h\in V(H)$, then $R$ is cyclically reduced.
\end{lemma}
\begin{proof}
	Suppose to the contrary that $R=e \cat R' \cat e^{-1}$ for some $e=(k_0,k_1) \in E(K)$ (where $k_0=\mu(g_0,h_0)$) and $R' \in \fundpi{k_1}{K}$.
	Let $C$ be any closed walk from $(g_0,h_0)$ in $G\times H$ of odd length (it exists by assumptions on $G$ and $H$).
	Then $\red{\mu(C)}$ must be odd too, so in particular $\red{\mu(C)}=\pexp{R}{i}$ for some $i\in\ZZ$ other than 0.
	Thus the first edge of $\red{\mu(C)}$ is $e$.
	Let $C=W_1 \cat W_2$, where $W_1$ is the longest prefix of $C$ such that $\red{\mu(W_1)}=e$.
	Let $(g',h')$ be the last vertex of $W_1$ (and first of $W_2$).
	For any closed walk $C'$ from $(g',h')$ in $G\times H$, since $W_1 \cat C' \cat W_1^{-1}$ is a closed walk from $(g_0,h_0)$ in $G\times H$, we have $\red{\mu(C')}=\red{\mu(W_1)}^{-1} \cdot \red{\mu(W_1 \cat C' \cat W_1^{-1})} \cdot \red{\mu(W_1)} = e^{-1} \cdot \pexp{R}{j} \cdot e = e^{-1} \cdot \pexp{(e \cdot R' \cdot e^{-1})}{j} \cdot e = \pexp{R'}{j}$ for some $j\in\ZZ$.
	Therefore the premises of Lemma~\ref{lem:recolored} are true for $(g',h')$ and $R'$ too (instead of $(g_0,h_0)$ and $R$).
	The first outcome of the lemma does not hold by assumption, so the second outcome is true, implying in particular that $\red{\mu(W_1^{-1})}$ is a prefix of $\pexp{R'}{k}$ for some $k\in\ZZ$.
	However, $\red{\mu(W_1^{-1})} = e^{-1}$ and this cannot be the first edge of $R'$ nor $R'^{-1}$, because $R=e \cat R' \cat e^{-1}$ is a reduced walk, a contradiction.
\end{proof}

The next lemma (used for $F=G\times H$) gives the final conclusion of the second outcome of Lemma~\ref{lem:recolored}.
The proof describes the homomorphisms and then just checks their validity.

\begin{lemma}\label{lem:cyclicCase}
	Let $\mu: F \to K$.
	Suppose there is an $R\in\fundpi{\mu(f_0)}{K}$ such that $R$ is cyclically reduced and for every closed walk $C$ from $f_0$ in $F$, $\red{\mu(C)} =\pexp{R}{i}$ for some $i\in\ZZ$.
	Suppose for any walk $W$ in $F$ starting at $f_0$, $\red{\mu(W)}$ is a prefix of $\pexp{R}{i}$ for some $i\in\ZZ$.
	Then there exist graph homomorphisms $\gamma:F \to C_{|R|}$ and $\delta : C_{|R|} \to K$ such that $\mu= \delta \circ \gamma$.
\end{lemma}
\begin{proof}
	\def\sgn{\mbox{sgn}}
	For a walk $W$ in $F$ starting at $f_0$, let $i\in\ZZ$ be such that $\red{\mu(W)}$ is a prefix of $\pexp{R}{i}$ and define $\gamma(W) = \sgn(i) \cdot |\red{\mu(W)}| \mod |R|$.
	Note this is unambiguous, as $\red{\mu(W)}$ either has zero length (so the choice of $i$ is irrelevant), or cannot be both a prefix of $\pexp{R}{i}$ for positive and negative $i$, because we assumed $R$ is cyclically reduced (so $\sgn(i)$ does not depend on the choice of $i$).
	
	For any two walks $W,W'$ from $f_0$ to the same endpoint, we want to show that $\gamma(W)=\gamma(W')$.
	Indeed, $\red{\mu(W')}=\red{\mu(W')}\cdot \red{\mu(W)}^{-1} \cdot \red{\mu(W)} = \red{\mu(W' \cat W^{-1})} \cdot \red{\mu(W)} = R^i \cdot \red{\mu(W)}$ for some $i\in\ZZ$ (since $W'\cat W^{-1}$ is a closed walk).
	Then one of the following holds, in each case implying $\gamma(W)=\gamma(W')$:
	\begin{itemize}
		\item $\red{\mu(W)}$ is empty, and then $\red{\mu(W')}=\pexp{R}{i}$ has length $0 \mod |R|$ too;
		\item $i=0$, implying $\red{\mu(W')}=\red{\mu(W)}$ and hence $\gamma(W)=\gamma(W')$ trivially;
		\item $\red{\mu(W)}$ is a prefix of $\pexp{R}{j}$ for $j\in\ZZ \setminus \{0\}$ with $\sgn(j)=\sgn(i)$, in which case $\red{\mu(W')}=\pexp{R}{i} \cdot \red{\mu(W)} = \pexp{R}{i} \cat \red{\mu(W)}$, which is a prefix of $\pexp{R}{i+j}$ with the same sign and length mod $|R|$; 
		\item $\red{\mu(W)}$ is a prefix of $\pexp{R}{j}$ for $j\in\ZZ \setminus \{0\}$ with $\sgn(j)=-\sgn(i)$ and $|\red{\mu(W)}|>|\pexp{R}{i}|$, in which case $\red{\mu(W')}= \pexp{R}{i} \cdot \red{\mu(W)}$  is a prefix of $\pexp{R}{j}$ of length $|\red{\mu(W)}|-|\pexp{R}{i}| = |\red{\mu(W)}| \mod |R|$;		
		\item $\red{\mu(W)}$ is a prefix of $\pexp{R}{j}$ for $j\in\ZZ \setminus \{0\}$ with $\sgn(j)=-\sgn(i)$ and $|\red{\mu(W)}| \leq |\pexp{R}{i}|$, in which case $\red{\mu(W')}= \pexp{R}{i} \cdot \red{\mu(W)}$ is a prefix of $\pexp{R}{i}$ of length $|\pexp{R}{i}| - |\red{\mu(W)}| = - |\red{\mu(W)}| \mod |R|$.
	\end{itemize}
	
	Therefore, we can unambiguously define $\gamma: V(F) \to \{0,\dots,|R|-1\}$ as $\gamma(f)=\gamma(W)$ for any walk $W$ from $f_0$ to $f$.
	This is a graph homomorphism from $F$ to $C_{|R|}$, because if $\{f,f'\}$ is an edge of $F$, then $\gamma(f')=\gamma(W \cat ff')=\gamma(W)\pm 1$ for any walk $W$ from $f_0$ to $f$.
	The last equality holds because $|\red{\mu(W\cat ff')}|=|\red{\mu(W)}|\pm 1$ and the sign in the definition of $\gamma$ can only change when one of $\red{\mu(W\cat ff')}, \red{\mu(W)}$ is empty.
	
	Let $R=r_0r_1 \cat r_1 r_2 \cat \dots \cat  r_{|R|-1} r_0$ for $r_i \in V(K)$. Define $\delta: \{0,\dots,|R|-1\} \to V(K)$ as $\delta(i) = r_i$.
	Since $R$ is a closed walk in $K$, $\delta:C_{|R|} \to K$ is a graph homomorphism.
	It is easily checked from definitions that $r_{\gamma(W)}$ is the endpoint of $\red{\mu(W)}$ for any walk $W$ from $f_0$ to $f$, and thus $\delta(\gamma(f))=\mu(f)$ for $f\in V(F)$.
\end{proof}

We are now ready to conclude the main theorem, in a slightly stronger form.
It gives $G\to K$, $H\to K$, or homomorphisms $G\times H \to C_n$ and $C_n \to K$.
By multiplicativity of cycles, the latter implies $G\to C_n \to K$ or $H\to C_n \to K$, concluding the proof of multiplicativity of square-free $K$.

\begin{theorem}\label{thm:main}
	Let $\mu:G\times H\to K$ for $K$ square-free, $G$ and $H$ connected and non-bipartite.
	Then there are graph homomorphisms $\mu^*: G\times H \to K$, $\gamma:G\times H\to I$ and $\delta:I\to K$ such that $\mu^*=\delta\circ\gamma$ and $\mu^*$ is reachable from $\mu$ by recoloring, where $I$ is either $G$, $H$ or $C_n$ for some $n\in\NN$.
\end{theorem}
\begin{proof}
	Let $\{(g_0,h_0),(g_1,h_1)\}$ be an edge of $G\times H$.
	By Theorem~\ref{thm:commute}, one of the following holds:
	\begin{itemize}
		\item $\red{\mu(C)}=\eps$ for every closed walk $C$ from $(g_0,h_0)$ in $G\times h_0 h_1$.
		Then, by repeatedly applying Lemma~\ref{lem:recoloring}, we can recolor $\mu$ to eventually reach a $K$-coloring $\mu^*$ with no $H$-extremal sets.
		Since it is reached by recoloring, Lemma~\ref{lem:invariant} guarantees that it still has the same property: $\red{\mu^*(C)}=\eps$ for every closed walk $C$ from $(g_0,h_0)$ in $G\times h_0 h_1$.
		
		Therefore $\mu^*|_{V(G\times H')} : G \times H' \to K$ satisfies the conditions of Lemma~\ref{lem:recolored} for $R=\eps$, $G$ and $H'=h_0 h_1$ (a graph isomorphic to $K_2$).
		The second outcome of the lemma cannot hold, because the reduced image of a one-edge walk $\red{\mu^*((g_0,h_0)(g_1,h_1))}$ has odd length and thus cannot be a prefix of $\pexp{\eps}{i}$ (the empty walk) for any $i\in\ZZ$.
		Hence the first outcome is true, that is, $\mu^*$ is constant on $V(G)\times\{h\}$ for some $h\in \{h_0,h_1\}$.
		The claim then follows for $I=H$ from Lemma~\ref{lem:recoloredConstant}.
		
		\item  $\red{\mu(D)}=\eps$ for every closed walk $D$ from $(g_0,h_0)$ in $g_0 g_1 \times H$.
		This case is entirely symmetric with the previous one, swapping the roles of $G$ and $H$. 
		The claim then follows for $I=G$.
		
		\item There is an $R\in \fundPi{K}$ such that for every closed walk $C$ from $(g_0,h_0)$ in $G\times H$, $\red{\mu(C)}=\pexp{R}{i}$ for some $i\in\ZZ$.
		Then, again by repeatedly applying Lemma~\ref{lem:recoloring}, we can recolor $\mu$ to eventually reach a $K$-coloring $\mu^*$ with no $H$-extremal sets.
		Since it is reached by recoloring, Lemma~\ref{lem:invariant} guarantees
		that there is an $R'\in \fundPi{K}$ such that for every closed walk $C$ from $(g_0,h_0)$ in $G\times H$, $\red{\mu^*(C)}=\pexp{R'}{i}$ for some $i\in\ZZ$.
		
		Hence $\mu^*:G\times H\to K$ and $R'$ satisfy the conditions of
		 Lemma~\ref{lem:recolored} directly.
		If the first outcome of the lemma holds, then the claim follows for $I=G$ from Lemma~\ref{lem:recoloredConstant}.
		Otherwise the second outcome is true, that is, for every walk $W$ in $G\times H$ starting from $(g_0,h_0)$, $\red{\mu(W)}$ is a prefix of $R'^i$ for some $i\in\ZZ$. By Lemma~\ref{lem:cyclicallyReducedR}, $R'$ is cyclically reduced.
		Then the claim follows for $I=C_{|R'|}$ from Lemma~\ref{lem:cyclicCase}.
		\qedhere
\end{itemize}
\end{proof}

%% file: figExtremal.tex
\begin{tikzpicture}[xscale=0.93]

\node at (0,0.6) {$h_0$};
\node at (2,0.6) {$h_1$};

\draw[rounded corners=5pt,draw=none,fill=blue!10,path fading=east] (-0.2,-3.9) -- (-0.2,-0.5) -- (1.8,0) -- (1.8,-5) -- cycle;
\draw[rounded corners=5pt,draw=none,fill=red!20,path fading=west] (2.2,-1)--(2.2,-3.6) -- (0.2,-3.9) -- (0.2,-0.5) -- cycle;

\draw[rounded corners=5pt] (1.8,0) -- (2.2,0) -- (2.2,-5) -- (1.8,-5) -- cycle;

\draw[rounded corners=5pt,fill=red!55] (1.8,-1) -- (2.2,-1) -- (2.2,-3.6) -- (1.8,-3.6) -- cycle;
\node[v,black,fill=red!60!black] at (2,-1.3) {};
\node[v,black,fill=red!60!black] at (2,-1.7) {};
\node at (2,-2.1) {$S$};
\node[v,black,fill=red!60!black] at (2,-2.5) {};
\node[v,black,fill=red!60!black] at (2,-2.9) {};
\node[v,black,fill=red!60!black] at (2,-3.3) {};
\draw [decorate,decoration={brace,amplitude=5pt,mirror},xshift=0pt,yshift=0pt] (1.72,-1.1)--(1.72,-3.5) node [midway,black!60,xshift=-10pt,yshift=-14pt,align=center] {$a$\\\rotatebox[origin=c]{-90}{$\rightsquigarrow$}\\$a'$};
\node at (1.2,-4.9) {$N^2(S)$};

\draw[rounded corners=5pt,fill=blue!35] (-0.2,-0.5) -- (0.2,-0.5) -- (0.2,-3.9) -- (-0.2,-3.9) -- cycle;
\node[v,black,fill=blue!60] at (0,-0.8) {};
\node[v,black,fill=blue!60] at (0,-1.2) {};
\node[v,black,fill=blue!60] at (0,-1.6) {};
\node[v,black,fill=blue!60] at (0,-2.0) {};
\node[v,black,fill=blue!60] at (0,-2.4) {};
\node[v,black,fill=blue!60] at (0,-2.8) {};
\node[v,black,fill=blue!60] at (0,-3.2) {};
\node[v,black,fill=blue!60] at (0,-3.6) {};
\draw [decorate,decoration={brace,amplitude=5pt,mirror},xshift=0pt,yshift=0pt] (-0.26,-0.6)--(-0.26,-3.8) node [midway,black!60,xshift=-10pt] {$b$};
\node at (0,-4.2) {$N(S)$};

\node[black!60] at (1.6,-0.3) {$a'$};
\node[v,fill=violet!40] at (2,-0.4) {};
\node[v,fill=violet!40] at (2,-3.9) {};
\node[v,fill=pink!40] at (2,-4.3) {};
\node[v,fill=violet!40] at (2,-4.7) {};

\begin{scope}[shift={(4.2,0)}]
\node at (0,0.6) {$h_0$};
\node at (2,0.6) {$h_1$};
\node at (4,0.6) {$h_2$};

\draw[rounded corners=5pt,draw=none,fill=blue!10,path fading=east] (-0.2,-3.9) -- (-0.2,-0.5) -- (1.8,0) -- (1.8,-5) -- cycle;
\draw[rounded corners=5pt,draw=none,fill=red!20,path fading=west] (2.2,-1)--(2.2,-3.6) -- (0.2,-3.9) -- (0.2,-0.5) -- cycle;
\draw[rounded corners=5pt,draw=none,fill=red!10,path fading=east] (1.8,-1)--(1.8,-3.6) -- (3.8,-3.9) -- (3.8,-0.5) -- cycle;

\draw[rounded corners=5pt] (1.8,0) -- (2.2,0) -- (2.2,-5) -- (1.8,-5) -- cycle;

\draw[rounded corners=5pt,fill=red!55] (1.8,-1) -- (2.2,-1) -- (2.2,-3.6) -- (1.8,-3.6) -- cycle;
\node[v,black,fill=red!60!black,label={[xshift=-1pt,yshift=-2pt]15:$(g'_1,h_1)$}] (red3) at (2,-1.3) {};
\node[v,black,fill=red!60!black] (r) at (2,-1.7) {};
\node at (2,-2.1) {$S$};
\node[v,black,fill=red!60!black] at (2,-2.5) {};
\node[v,black,fill=red!60!black] at (2,-2.9) {};
\node[v,black,fill=red!60!black] at (2,-3.3) {};
\node at (2.8,-4.9) {$N^2(S)$};

\draw[rounded corners=5pt,fill=blue!35] (-0.2,-0.5) -- (0.2,-0.5) -- (0.2,-3.9) -- (-0.2,-3.9) -- cycle;
\node[v,black,fill=blue!60] (yh0) at (0,-0.8) {};
\node[v,black,fill=blue!60] at (0,-1.2) {};
\node[v,black,fill=blue!60] (gp) at (0,-1.6) {};
\node[v,black,fill=blue!60] (yp) at (0,-2.0) {};
\node[v,black,fill=blue!60] at (0,-2.4) {};
\node[v,black,fill=blue!60,label={[yshift=0pt,xshift=-1pt]0:$(g_2,h_0)$}] (bluen) at (0,-2.8) {};
\node[v,black,fill=blue!60] at (0,-3.2) {};
\node[v,black,fill=blue!60] at (0,-3.6) {};
\draw [decorate,decoration={brace,amplitude=5pt,mirror},xshift=0pt,yshift=0pt] (-0.26,-0.6)--(-0.26,-3.8) node [midway,black!60,xshift=-10pt] {$b$};
\node at (0,-4.2) {$N(S)$};

\node[black!60] at (1.6,-0.3) {$a'$};
\node[black!60] at (1.6,-1.2) {$a$};
\node[v,fill=violet!40] (xh1) at (2,-0.4) {};
\node[v,fill=violet!40,label={[xshift=-1pt,yshift=-2pt]0:$(g_1,h_1)$}] (violet1) at (2,-3.9) {};
\node[v,fill=pink!40] at (2,-4.3) {};
\node[v,fill=violet!40] at (2,-4.7) {};

\draw[rounded corners=5pt,draw=none,fill=yellow!35,path fading=west] (4.2,-1.8) -- (4.2,-3.0) -- (2.2,-3.2) -- (2.2,-1.3) -- cycle;
\draw[rounded corners=5pt] (3.8,-0.5) -- (4.2,-0.5) -- (4.2,-3.9) -- (3.8,-3.9) -- cycle;
\draw[rounded corners=5pt,gray,fill=yellow!25] (3.8,-1.8) -- (4.2,-1.8) -- (4.2,-3.0) -- (3.8,-3.0) -- cycle;
\node[v,fill=blue!60] (yh2) at (4,-0.8) {};
\node[v,fill=green!60] at (4,-1.2) {};
\node[v,fill=green!60] (g) at (4,-1.6) {};
\node[v,black!60,fill=yellow!40,label={[yshift=0pt]180:$S'$}] (y) at (4,-2.0) {};
\node[v,black!60,fill=yellow!40] at (4,-2.4) {};
\node[v,black!60,fill=yellow!40,label={[yshift=-1pt,xshift=1pt]180:$(g_2,h_2)$}] (botyellow) at (4,-2.8) {};
\node[v,fill=green!60] at (4,-3.2) {};
\node[v,fill=blue!60] at (4,-3.6) {};
\draw [decorate,decoration={brace,amplitude=3pt},xshift=0pt,yshift=0pt] (4.27,-1.9) -- (4.27,-2.9)
node [midway,black!60,xshift=7pt,yshift=-2pt] {$b'$};
\end{scope}

\draw (bluen) -- (red3) -- (botyellow);
\draw[dashed] (botyellow) -- (violet1)  -- (bluen);

\begin{scope}[shift={(10.8,0)}]
\node at (0,0.6) {$h_0$};
\node at (2,0.6) {$h_1$};
\node at (4,0.6) {$h_2$};

\draw[rounded corners=5pt,draw=none,fill=blue!10,path fading=east] (-0.2,-3.9) -- (-0.2,-0.5) -- (1.8,0) -- (1.8,-5) -- cycle;
\draw[rounded corners=5pt,draw=none,fill=red!20,path fading=west] (2.2,-1)--(2.2,-3.6) -- (0.2,-3.9) -- (0.2,-0.5) -- cycle;
\draw[rounded corners=5pt,draw=none,fill=red!10,path fading=east] (1.8,-1)--(1.8,-3.6) -- (3.8,-3.9) -- (3.8,-0.5) -- cycle;

\draw[rounded corners=5pt] (1.8,0) -- (2.2,0) -- (2.2,-5) -- (1.8,-5) -- cycle;

\draw[rounded corners=5pt,fill=red!55] (1.8,-1) -- (2.2,-1) -- (2.2,-3.6) -- (1.8,-3.6) -- cycle;
\node[v,black,fill=red!60!black,label={15:$(z,h_1)$}] (zh1) at (2,-1.3) {};
\node[v,black,fill=red!60!black] (r) at (2,-1.7) {};
\node at (2,-2.1) {$S$};
\node[v,black,fill=red!60!black] at (2,-2.5) {};
\node[v,black,fill=red!60!black] at (2,-2.9) {};
\node[v,black,fill=red!60!black] at (2,-3.3) {};
\draw [decorate,decoration={brace,amplitude=5pt,mirror},xshift=0pt,yshift=0pt] (1.72,-1.1)--(1.72,-3.5) node [midway,black!60,xshift=-10pt,yshift=-14pt,align=center] {$a$\\\rotatebox[origin=c]{-90}{$\rightsquigarrow$}\\$a'$};
\node at (2.8,-4.9) {$N^2(S)$};

\draw[rounded corners=5pt,fill=blue!35] (-0.2,-0.5) -- (0.2,-0.5) -- (0.2,-3.9) -- (-0.2,-3.9) -- cycle;
\node[v,black,fill=blue!60,label={60:$(y,h_0)$}] (yh0) at (0,-0.8) {};
\node[v,black,fill=blue!60] at (0,-1.2) {};
\node[v,black,fill=blue!60] (gp) at (0,-1.6) {};
\node[v,black,fill=blue!60] (yp) at (0,-2.0) {};
\node[v,black,fill=blue!60] at (0,-2.4) {};
\node[v,black,fill=blue!60] at (0,-2.8) {};
\node[v,black,fill=blue!60] at (0,-3.2) {};
\node[v,black,fill=blue!60] at (0,-3.6) {};
\draw [decorate,decoration={brace,amplitude=5pt,mirror},xshift=0pt,yshift=0pt] (-0.26,-0.6)--(-0.26,-3.8) node [midway,black!60,xshift=-10pt] {$b$};
\node at (0,-4.2) {$N(S)$};

\node[black!60] at (1.7,-0.1) {$a'$};
\node[v,fill=violet!40,label={15:$(x,h_1)$}] (xh1) at (2,-0.4) {};
\node[v,fill=violet!40] at (2,-3.9) {};
\node[v,fill=pink!40] at (2,-4.3) {};
\node[v,fill=violet!40] at (2,-4.7) {};

\draw[rounded corners=5pt,draw=none,fill=yellow!35,path fading=west] (4.2,-1.8) -- (4.2,-3.0) -- (2.2,-3.2) -- (2.2,-1.3) -- cycle;
\draw[rounded corners=5pt] (3.8,-0.5) -- (4.2,-0.5) -- (4.2,-3.9) -- (3.8,-3.9) -- cycle;
\draw[rounded corners=5pt,gray,fill=yellow!25] (3.8,-1.8) -- (4.2,-1.8) -- (4.2,-3.0) -- (3.8,-3.0) -- cycle;
\node[v,fill=blue!60,label={60:$(y,h_2)$}] (yh2) at (4,-0.8) {};
\node[v,fill=green!60] at (4,-1.2) {};
\node[v,fill=green!60,label={0:$(y',h_2)$}] (g) at (4,-1.6) {};
\node[v,black!60,fill=yellow!40] (y) at (4,-2.0) {};
\node[v,black!60,fill=yellow!40,label=0:$S'$] at (4,-2.4) {};
\node[v,black!60,fill=yellow!40] at (4,-2.8) {};
\node[v,fill=green!60] at (4,-3.2) {};
\node[v,fill=blue!60] at (4,-3.6) {};
\draw [decorate,decoration={brace,amplitude=4pt,mirror},xshift=0pt,yshift=0pt] (3.73,-1.9) -- (3.73,-2.9)
node [midway,black!60,xshift=-10pt] {$b'$};

\draw (xh1) -- (yh2) -- (zh1) -- (yh0) -- (xh1);
\draw[black!50] (y) -- (r) -- (g) -- (zh1);
\draw[black!50] (yp) -- (r) -- (gp) -- (zh1);
\end{scope}

\end{tikzpicture}

%% file: figExtension.tex
\begin{tikzpicture}

\draw[rotate around={60:(-30:3.97)},draw=none,fill=red!60] (-30:3.97) ellipse (10pt and 5pt);
\draw[rotate around={60:(-30:1.8)},draw=none,fill=blue!30] (-30:1.8) ellipse (10pt and 10pt);
\draw[rotate around={60:(-150:1.9)},draw=none,fill=violet!20] (-150:1.9) ellipse (12pt and 12pt);
\draw[rotate around={60:(70:1.8)},draw=none,fill=pink!20] (70:1.8) ellipse (9pt and 9pt);

\draw[rounded corners = 3pt,-{latex},thick,
		decoration={markings, mark=at position 0.5 with {\arrow{>}}},
		postaction={decorate}
	]  (209.5:4) -- (207.5:1.66) arc (207.5:-147.5:1.66) node[midway,below left] {$R$} --  (-149.5:4);
\draw (0,0);
\draw[rounded corners = 3pt,-{latex},
		decoration={markings, mark=at position 0.5 with {\arrow{>}}},
		postaction={decorate}
	] (209:4) -- (202:1.9) arc (202:-26:1.9) node[midway,above]{$\overline{\mu(W)}$}-- (-29.5:4);
\draw[rounded corners = 3pt,-{latex},
		decoration={markings, mark=at position 0.5 with {\arrow{>}}},
		postaction={decorate}
	] (-149:4) -- (-142:1.9) arc (-142:-33:1.9) node[midway,below]{$\overline{\mu(W')}$} -- (-30.5:4);
\node at (-152:4.6) {$(g_0,h_0)$};

\draw [decorate,decoration={brace,amplitude=5pt},xshift=3pt,yshift=2pt] (-26:2.0)--(-28:3.9) node[midway,above,yshift=3pt,xshift=10pt] {$\texttt{ext}(W)$};

\draw [decorate,decoration={brace,amplitude=5pt,mirror},xshift=-2pt,yshift=-3pt] (-33:2.0)--(-32:3.9) node[midway,below,xshift=-10pt,yshift=-5pt] {$\texttt{ext}(W')$};

\node at (-30:4.6) {$(g^*,h^*)$};
\end{tikzpicture}

%% file: concls.tex
Some further conclusions can be drawn from Theorem~\ref{thm:main}.
For one example, let $\mu:G\times H \to K$ for a square-free graph $K$, and suppose that $G\not\to K$, $G\not \to K_3$, $H\not\to K_3$, and that $H$ has only one $K$-coloring $\gamma$, up to automorphisms of $K$.
Then $\mu$ is the only $K$-coloring of $G\times H$ (up to automorphisms of $K$).
Indeed, the only possible outcome of Theorem~\ref{thm:main} (again up to automorphisms) is that $\mu$ can be recolored to $\mu^* = \delta \circ \gamma$, where $\delta$ is the projection to $H$.
That is, $\mu^*$ is constant on $V(G)\times\{h\}$ for each $h\in V(H)$.
If $\mu^*\neq \mu$, then it was obtained by recoloring; let the last recoloring step change the color of $(g,h)\in V(G\times H)$ from $a \in V(K)$ to $b:=\mu^*(g,h)$.
Before this step, the $K$-coloring was still constant on $V(G)\times\{h'\}$ for $h'\neq h \in V(H)$.
Hence replacing all values of $\mu^*(\cdot, h)=b$ with $a$ gives a different $K$-coloring of $G\times H$, which is a composition of the projection to $H$ with a different $K$-coloring of $H$ (but different only on $g$).
But this is impossible, thus in fact $\mu=\mu^*$ is the only $K$-coloring of $G\times H$ (up to automorphisms of $K$).

Second, multiplicativity of square-free graphs $K$ can be strengthened to the following statement: 
for graphs $G,H$ and odd cycles $G'$ in $G$ and $H'$ in $H$, if $G\times H' \cup G' \times H$ (an induced subgraph of $G\times H$) has a $K$-coloring, then $G\to K$ or $H\to K$.
This follows by adapting Theorem~\ref{thm:commute} so that depending on the types of $G'$ and $H'$ we have one of the same three conclusions, with the first two limited to closed walks in $G' \times h_0 h_1$ and $g_0 g_1 \times H'$ respectively (instead of $G \times h_0 h_1$ and $g_0g_1\times H$).
The third case is without change.
In the first two, we then consider $G'$ instead of $G$ (or $H'$ instead of $H$, respectively) and continue the proof without change (applying Lemma~\ref{lem:recolored} to $G'\times h_0h_1$ only) to eventually get $H\to K$ (or $G\to K$, respectively).
In other words, if $\red{\mu(\lcycle{C}{h_0 h_1})}=\eps$ for one odd cycle $C$ of $G$, then this already implies $H\to K$.

Unfortunately, this means our methods have the same limitations as previous ones: Tardif and Zhu~\cite{TardifZ02} showed that an analogous extension is false for $K=K_n$ with $n\geq 4$. Namely, for any $m>n\geq 4$ there exists $m$-chromatic graphs $G,H$ with $n$-chromatic subgraphs $G',H'$ such that $G\times H' \cup G'\times H$ is $n$-chromatic.
On the other hand, our approach yields results for high-chromatic graphs while still only relying on cycles, essentially.

Let us also mention that the proofs here are constructive, in the sense that given a $K$-coloring of $G\times H$, a $K$-coloring of $G$ or $H$ can be found in polynomial time.
This is straightforward for circular cliques, while for square-free graphs this follows from the fact that a $K$-coloring of $G\times H$ can be $H$-improved only polynomially many times.
More explicit colorings, for example describing colors of nodes of the exponential graphs $K^G$ in time polynomial in $G$, remain an interesting open problem, see~\cite{Tardif06}.